\documentclass[reqno]{amsart}
\usepackage{amsmath,xypic,mathrsfs}
\usepackage[foot]{amsaddr}
\usepackage{hyperref}
\hypersetup{
 pdftitle={ Superconformal structures on GCY metric manifolds},
    pdfauthor={R.~Heluani \and M.~Zabzine}
    pdfnewwindow=true,      
    colorlinks=true,       
    linkcolor=black,          
    citecolor= black,        
    filecolor= black,      
    urlcolor= black           
}

\numberwithin{equation}{section}
\newtheorem{prop}{Proposition}
\newtheorem{lem}{Lemma}
\newtheorem{thm}{Theorem}
\newtheorem{cor}{Corollary} 
\newtheorem{rem}{Remark}

\theoremstyle{definition}
\newtheorem{defn}{Definition}
\newtheorem{ex}{Example}

\def\cH{\mathcal{H}}
\def\cR{\mathcal{R}}
\def\cO{\mathcal{O}}
\def\vac{|0\rangle}
\DeclareMathOperator{\End}{End}
\DeclareMathOperator{\Id}{Id}
\DeclareMathOperator{\dive}{div}
\DeclareMathOperator{\lie}{Lie}
\DeclareMathOperator{\vol}{vol}
\begin{document}
\title[Superconf. structures on generalized Calabi-Yau metric manifolds]{Superconformal structures on\\
 Generalized Calabi-Yau metric manifolds}
\author{Reimundo Heluani$^1$}
\address{$^1$Department of Mathematics, University of California, Berkeley, CA 94720, USA}

\author{Maxim Zabzine$^2$}
\address{$^2$Department of Physics and Astronomy, 
     Uppsala university,
     Box 516, SE-751\;20 Uppsala,
     Sweden}

\email{heluani@math.berkeley.edu, Maxim.Zabzine@fysast.uu.se}
\begin{abstract}
We construct an embedding of two commuting copies of the $N=2$ superconformal vertex algebra in the space of global sections of the twisted \emph{chiral-anti-chiral de Rham complex} of a generalized Calabi-Yau metric manifold, including the case when there is a non-trivial $H$-flux and non-vanishing dilaton. The $4$ corresponding BRST charges are well defined on any generalized K\"ahler manifold. This allows one to consider the \emph{half-twisted} model defining thus the \emph{chiral de Rham complex} of a generalized K\"ahler manifold. The classical limit of this result allows one to recover the celebrated \emph{generalized K\"ahler identities} as the degree zero part of an infinite dimensional Lie superalgebra attached to any generalized K\"ahler manifold.  As a byproduct of our study we investigate the properties of generalized Calabi-Yau metric manifolds in the Lie algebroid setting. 

\end{abstract} 
\maketitle

\section{Introduction}
On a K\"ahler manifold one has a decomposition of the sheaf of differential forms into a bi-complex $(\wedge^{p,q} T^*, \partial, \overline{\partial})$. An analog of such a decomposition exists for any generalized K\"ahler manifold \cite{gualtieri3}. The aim of this article is to provide an \emph{affine} or \emph{chiral} analog of this result in the case when the manifold is generalized Calabi-Yau metric, extending thus the results in \cite{heluani8} and producing the quantum counterpart of the results in \cite{Bredthauer:2006hf, MR2322403}.

To any differentiable manifold $M$ one can associate a sheaf of vertex algebras $\mathrm{CDR}(M)$ \cite{malikov}. More generally, given any Courant algebroid $E$ one constructs a sheaf of SUSY vertex algebras $U^\mathrm{ch}(E)$ \cite{heluani8}. When $E$ is endowed with a \emph{generalized Calabi-Yau} structure, there is an embedding of the $N=2$ superconformal vertex algebra into the global sections of $U^\mathrm{ch}(E)$ \cite{heluani9}. In the usual Calabi-Yau case, it was shown in \cite{heluani8} that one can in fact construct two commuting copies of the $N=2$ superconformal structure, each with central charge $\tfrac{3}{2}\mathrm{dim} M$. In this article we combine and generalize these results to the case when $E$ is endowed with a \emph{generalized Calabi-Yau metric structure} as defined in \cite{gualtieri1}. 

An interesting new phenomenon in this article is that in the presence of a non-trivial $H$-flux, the \emph{dilaton} field plays a crucial role in all of our formulas. This feature is well known in the physics literature.   

On a given K\"ahler manifold $M$ with Hermitian metric $g$, the existence of a global holomorphic volume form $\Omega$ is intimately related with the vanishing of the Ricci curvature of $g$ and with the fact that the holonomy of $M$ reduces to $SU(n)$. The metric $g$ gives rise to a volume form $\vol_g$ and the global holomorphic volume form satisfies $\Omega \wedge \overline{\Omega} = \vol_g$. The volume form $\Omega$ is covariantly constant with respect to the Levi-Civita connection of $g$ and the vanishing of the Ricci tensor is expressed in holomorphic coordinates by $\partial_\alpha \partial_{\bar \beta} \log \sqrt{ \det g}= 0$. In the generalized K\"ahler case the situation is subtler. There exists a dictionary between generalized K\"ahler manifolds and  bihermitian manifolds \cite{gualtieri1}.  The latter are  bihermitian manifolds $(M, J_\pm, g)$ with  two connections $\nabla^\pm$  with torsion encoded by a closed three form and such that $\nabla^\pm J_\pm=0$. To the data of a generalized Calabi-Yau metric manifold we can associate two holomorphic volume forms $\Omega_\pm$ which are holomorphic with respect to $J_\pm$. These in turn give rise to a unique volume form $\nu=\Omega_\pm \wedge \overline{\Omega_\pm}$ and the ratio between this volume form and the Riemannian volume form defines the dilaton $\Phi$ by $\nu = e^{-4\Phi}\vol_g$. It is not the holomorphic volume forms $\Omega_\pm$ that enter in the fields of the $N=2$ structure, but rather the forms corrected by the dilaton $e^{-2\Phi} \Omega_\pm$ which become covariantly constant with respect to $\nabla^\pm$. The analogous statement to the vanishing of the Ricci tensor becomes \[ \partial_\alpha \partial_{\bar \beta} \left( \log \left( e^{-4 \Phi} \sqrt{\det g} \right) \right)= 0, \] where we use holomorphic coordinates for either complex structure. We show in Section \ref{sec:5} that these statements correspond to the \emph{unimodularity} of the Lie algebroids corresponding to the generalized complex structures $\mathcal{J}_{1,2}$. 

The existence of $N=2$ superconformal supersymmetry allows us to perform a topological twist and in particular consider the BRST cohomology. When we have two commuting copies of the superconformal algebra we may perform the topological twist in one of the two sectors, say the plus sector, and consider its BRST cohomology. Carrying out this construction on a generalized Calabi-Yau metric manifold, produces a sheaf of SUSY vertex algebras with a remaining $N=2$ superconformal structure (that of the minus sector). In the case when $M$ is a \emph{usual} Calabi-Yau manifold, this sheaf is isomorphic to the  \emph{chiral de Rham complex} of $M$ defined in \cite{malikov}, together with its topological structure.
Moreover, since in order to consider the BRST cohomology we need only the zero modes of fields to be well defined (as opposed to the full superconformal algebra) we may perform the above mentioned \emph{half-twisting} procedure to obtain a sheaf of SUSY vertex algebras on any generalized K\"ahler manifold $M$, we call this sheaf the \emph{chiral de Rham complex of $M$}, a name that is justified since in the usual K\"ahler case we recover the construction of \cite{malikov} in the holomorphic setting.

In the usual K\"ahler case, the \emph{holomorphic} chiral de Rham complex of $M$ can be described purely in terms of holomorphic data by generators and relations. The interpretation of this sheaf as a half-twisted model was given in \cite{witten} and in the supersymmetric setting in \cite{kapustin}. The situation in the generalized K\"ahler case is subtler since there is no obvious notion of what ``holomorphic data'' means. In Theorem \ref{thm:holomorphiccdr} below, we give such a description, after developing rudimentary notions of differential calculus on generalized K\"ahler manifolds. This result extends that of \cite{witten} \cite{kapustin} to the generalized K\"ahler case with or without $H$-flux, while at the same time we find an interesting new spin (see Remark \ref{rem:buenisimo}).  In the bihermitian setup one can attach (a twisted version of) the holomorphic chiral de Rham complex of \cite{malikov} to each one of the two Hermitian complex structures. We show that these sheaves agree with the ones constructed by BRST reduction.

The existence of two commuting conformal structures allows us to consider $U^\mathrm{ch}(E)$ as a formal Hamiltonian quantization of the sigma-model with target a generalized Calabi-Yau metric manifold. In general, consider a vertex algebra $V$ endowed with two commuting Virasoro fields $L^\pm(z)$. Suppose moreover that $L = L^+ + L^-$ is a \emph{conformal structure on $V$} \cite{kac:vertex} i.e. $L_0$ acts diagonally and $L_{-1} = T$, the translation operator on $V$. Consider the formal change of coordinates $z = e^{i\sigma}$ and the Hamiltonian 
\begin{equation} H = i \int \left( L^+ - L^- \right) d\sigma~.
\end{equation}
For any state $a \in V$, we can impose the equations of motion
\begin{equation}
\partial_\tau Y(a,\sigma) = [Y(a,\sigma),H]~,
\label{eq:0.2a}
\end{equation}
to obtain a state field correspondence $a \mapsto Y(a,\sigma,\tau)=Y(a,z,\bar{z})$, where $z = e^{i\sigma + \tau}$, and $\bar{z} = e^{i \sigma - \tau}$,
 so that the zero mode of $L^+$ acts as as $\partial_z$ and the zero mode of $L^-$ acts as $\partial_{\bar{z}}$. With these considerations, we obtain the equations of motion for the quantum non-linear sigma model with target a generalized Calabi-Yau metric manifold, very much 
  in analogy to standard Calabi-Yau story \cite{heluani10}.

The organization of this article is as follows. In section \ref{sec:1} we fix notations and briefly recall the definitions of SUSY vertex algebras. In section \ref{sec:2} we recall the basic definitions of generalized K\"ahler and Calabi-Yau metric manifolds. In section \ref{sec:3} we recall the construction of the sheaf of SUSY vertex algebras $U^\mathrm{ch}(E)$.  In section \ref{sec:4} we recall the connection with bihermitian geometry and we introduce the basic local coordinate frames that will play an important role in the computations in latter sections. In section \ref{sec:5} we collect some useful Lemmas about unimodularity in generalized Calabi-Yau metric manifolds, we collect some results scattered in the literature and produce some new ones. In particular, we clarify the connection between generalized Calabi-Yau metric manifolds as in \cite{gualtieri1} and their bihermitian counterpart. In section \ref{sec:6} we state and prove the main results of this article. In section \ref{sec:7} we study the topological twists and corresponding BRST cohomologies. We define here the chiral de Rham complex for a Generalized K\"ahler manifold. We develop in this section the rudiments of differential calculus on generalized K\"ahler manifolds and show that the chiral de Rham complex can be described entirely interms of \emph{holomorphic} data. In section \ref{sec:8} we present a brief summary and discussion of the results in the present article. 

{\bf Acknowledgements:} 
We thank Nigel Hitchin, Chris Hull, Ulf Lindstr\"om, Maciej Szczesny, Rikard von Unge and Frederik Witt for the discussions for this and related 
 subjects.  In particular we are grateful to Jian Qiu for inspiring discussions and a few useful suggestions. 
 The research of R.H. was supported by NSF grant DMS-0635607002.
The research of M.Z. is supported by VR-grant 621-2008-4273.

\section{Preliminaries on SUSY vertex algebras} \label{sec:1}
In this section we collect some
results on SUSY vertex algebras from \cite{heluani3}. 
\begin{defn}[\cite{heluani3}]
    An $N_K=1$ SUSY vertex algebra consists of the data of a vector space $V$,
    an even vector $\vac \in V$ (the vacuum vector), an odd endomorphism
    $S$ (whose square is an even endomorphism we denote $T$),
    and a parity preserving linear map $A \mapsto Y(A,z,\theta)$ from
     $V$ to $\End(V)$-valued fields (the state-field correspondence). This
     data should satisfy the following set of axioms:
    \begin{itemize}
	\item For any $A$, $Y(A,z,\theta)$ is a field, namely
	\[ Y(A,z,\theta) B \in V[ [z]][\theta], \qquad \forall B \in V~,\]
    \item Vacuum axioms:
        \begin{equation*}
            \begin{aligned}
            Y(\vac, z,\theta) &= \Id, \\
            Y(A, z,\theta) \vac &= A + O(z,\theta), \\
            S \vac &= 0.
            \end{aligned}
        \end{equation*}
    \item Translation invariance:
        \begin{equation*}
            \begin{aligned}
            {[} S, Y(A,z,\theta)] &= (\partial_\theta - \theta \partial_z)
            Y(A,z,\theta),\\
            {[}T, Y(A,z,\theta)] &= \partial_z Y(A,z,\theta).
        \end{aligned}
        \end{equation*}
    \item Locality:
        \begin{equation*}
            (z-w)^n [Y(A,z,\theta), Y(B,w,\zeta)] = 0, \qquad n \gg 0.
        \end{equation*}
     \end{itemize}
    \label{defn:2.3}
\end{defn}
    Given a $N_K=1$ SUSY vertex algebra $V$ and a vector $A \in V$, we expand the fields
    \begin{equation*}
        Y(A,z,\theta) = \sum_{\stackrel{j \in \mathbb{Z}}{J = 0,1}} Z^{-1-j|1-J}
        A_{(j|J)},
    \end{equation*}
    and we call the endomorphisms $A_{(j|J)}$ the \emph{Fourier modes} of
    $Y(A,Z)$. Define now the operations:
    \begin{equation}
        \begin{aligned}
            {[}A_\Lambda B] &= \sum_{\stackrel{j \geq 0}{J = 0,1}}
            \frac{\Lambda^{j|J}}{j!} A_{(j|J)}B, \\
            A B &= A_{(-1|1)}B.
        \end{aligned}
        \label{eq:2.4.2}
    \end{equation}
    The first operation is called the $\Lambda$-bracket and it encodes all the information in the OPE of the superfields $Y(A,z,\theta)$ and $Y(B,z,\theta)$. The second operation is
    called the \emph{normally ordered product}, the set of axioms that these operations satisfy are summarized in Appendix \ref{sec:appendixc}

\section{Preliminaries in geometry} \label{sec:2} 
In this section we recall the basic definitions of generalized complex geometry
following \cite{gualtieri1} and \cite{gualtieri2}. 

Let $M$ be a smooth manifold and denote by $T$ the tangent bundle of
$M$.
\begin{defn}
	A \emph{Courant algebroid} is a vector bundle $E$ over $M$, equipped with a
	nondegenerate symmetric bilinear form $\langle ,\rangle $ as
	well as a bilinear bracket $[,]$ on
	$C^\infty(E)$ and with a smooth bundle map $\pi: E
	\rightarrow  T$ called the anchor. 

	 	 These structures should satisfy the following five axioms 
	\begin{enumerate}
		\item $\pi([A,B]) = [\pi(A), \pi(B)], \quad \forall A, B
			\in C^\infty(E)$. 
		\item The bracket $[,]$ should satisfy the  Leibniz identity. 
			\[ [A, [B, C]] = [[A,B],C]+ [B, [A,C]], \quad \forall A, B, C \in
			C^\infty(E) \]
		\item $[A, f B] = f [A, B] + (\pi(A) f)B$, for all $A, B \in C^\infty(E)$ and
			$f \in C^\infty(M)$, 
		\item $ \langle A, [B, C] + [C,B] \rangle = \pi(A) \langle B, C\rangle, \quad \forall A, B, C \in
			C^\infty(E) $
		\item $\pi(A)\langle B, C\rangle  = \langle [A, B]  , C\rangle  + \langle B, [A, C]
			\rangle , \quad \forall A, B, C \in
			C^\infty(E)$. 
	\end{enumerate}
	We can introduce a natural
	 differential operator $\mathcal{D}: C^\infty(M) \rightarrow
	 C^\infty(E)$ as $\langle \mathcal{D}f, A\rangle  = \tfrac{1}{2} \pi(A) f$ for
	all $f \in C^\infty (M)$ and $A \in
	C^\infty(E)$. Thus property (4) becomes 
	$$ [B, C] + [C, B] = \mathcal{D} \langle B, C \rangle ~.$$
 Another useful identity implied by the definition is 	
	$\pi \circ \mathcal{D} = 0$, i.e. $\langle \mathcal{D} f,
			\mathcal{D} g\rangle  = 0,
			\quad \forall f,g \in C^\infty(M)$.  The bracket $[,]$ is called the \emph{Dorfman} bracket, 
    in some situations it is convenient to use the antisymmetric version, the Courant bracket $[~,~]_c$
     which is related to Dorfman bracket as follows 
     \begin{equation}
	[A, B] = [A, B]_c +
\mathcal{D} \langle X, Y \rangle.
\label{eq:dorf}
\end{equation}
A Courant algebroid $E$ is called \emph{exact} if the following
	sequence is exact:
	\[ 0 \rightarrow T^* \xrightarrow{\pi^*} E \xrightarrow{\pi} T
	\rightarrow 0, \]
	where we use the inner product in $E$ to identify it with its
	dual. In this case  it is possible to choose an isotropic splitting $s: T \rightarrow E$ for $\pi$
	 giving rise to an isomorphism $E \cong T \oplus T^*$  taking the Dorfman bracket to that given in the example below. 
	\label{defn:1}
\end{defn}
\begin{ex}
	$E=(T \oplus T^*)\otimes \mathbb{C}$, $\langle,\rangle$ and $[,]$ are respectively the natural symmetric pairing and
	the Dorfman bracket defined as:
	\begin{equation*}
		\begin{aligned} \langle X + \zeta, Y + \eta\rangle &= \frac{1}{2} \left( i_X \eta + i_Y\zeta
\right).  \\
 {[X + \zeta}, Y + \eta] &= [X, Y] + \mathrm{Lie}_X \eta -
 i_Y d\zeta + i_Y i_X H~,
\end{aligned}
\end{equation*}
\label{ex:courant1}
 where $H$ is a closed three form. 
\end{ex}
In the rest of this article all Courant algebroids will be assumed to be exact unless noted.
\begin{defn}[{\cite[4.14]{gualtieri1}}]
	A generalized almost complex structure on a real $2n$-dimensional manifold
	$M$ is given by the following equivalent data:
	\begin{itemize}
		\item an endomorphism $\mathcal{J}$ of $E$ which is
			orthogonal with respect to the inner product
			$\langle, \rangle$ and $\mathcal{J}^2 = - 1$.
		\item a maximal isotropic sub-bundle $L \subset E \otimes
			\mathbb{C}$ of real index zero, i.e. $L \cap \bar{L} = 0$. 
		\item a pure spinor line sub-bundle $U \subset \bigwedge^* T^* \otimes
			\mathbb{C}$, called the \emph{canonical line bundle}
			satisfyinng $(\varphi, \bar{\varphi}) \neq 0$ at each point
			$x \in M$ for any generator $\varphi \in U_x$. 
	\end{itemize}
	\label{defn:almost-complex}
\end{defn}
Here $(~,~): \wedge^* T^* \otimes \wedge^* T^* \rightarrow \det T^*$ is the Mukai pairing which  is an 
 invariant bilinear form on the spinors of $E \cong T\oplus T^*$ defined as
 \begin{equation*}
   (\varphi, \psi) \equiv [\varphi^\top \wedge \psi ]_{\rm top}~,
 \end{equation*}
  where $\varphi^\top$ denotes the antiautomorphism of the Clifford algebra applied to $\varphi$, 
   see \cite{gualtieri2} for an extensive explanation on the subject.

The fact that $L$ is of real index zero implies 
\begin{equation*}
	E\otimes \mathbb{C} \simeq (T \oplus T^*) \otimes \mathbb{C} = L \oplus \bar{L} = L \oplus L^*, 
\end{equation*}
using $\langle, \rangle$ to identify $\bar{L}$ with $L^*$.
\begin{defn}[{\cite[4.18]{gualtieri1}}]
	A generalized almost complex structure $\mathcal{J}$ is said to be
	integrable to a generalized complex structure when its $+i$-eigenvalue $L \subset
	E \otimes \mathbb{C}$ is Courant (Dorfman) involutive. 
	\label{defn:complex}
\end{defn}
 We refer to a manifold admitting integrable generalized complex structure as generalized complex manifold.
 
 \begin{prop}[{\cite{Lindstrom:2004iw,  gualtieri2}}] 
  Every generalized complex manifold is a Poisson manifold, i.e. it admits a bivector $P = P^{ij} \partial_i \wedge \partial_j$ such that 
   \begin{equation*}
   P^{ik} \partial_k P^{jl} +   P^{jk} \partial_k P^{li} +  P^{lk} \partial_k P^{ij} =0~.
 \end{equation*}
  We refer to such $P$ as Poisson structure. 
  \end{prop}

In this case, $(L,L^*)$ is a Lie bi-algebroid (that is, both $L$ and its dual $L^*$ are naturally Lie algebroids in a suitably compatible manner), and $E\otimes
\mathbb{C}$ could be viewed as its \emph{Drinfeld double}. Note that $E$ acts on the sheaf of differential forms
$\bigwedge^\bullet T^*$ via the spinor representation, and this sheaf acquires 
a different grading by the eigenvalues of $\mathcal{J}$ acting via the
spinor representation: 
\begin{equation} 
\bigwedge T^* = U_{-n} \oplus \dots \oplus
U_{n}.
\label{eq:ukn}
\end{equation}
	Clifford multiplication by sections of $\overline{L}$
(resp. $L$)
increases (resp. decreases) the grading. $U_{-n}= U_{\mathcal{J}}$ is called the \emph{canonical bundle}
of $(M, \mathcal{J})$.
\begin{defn}
	A generalized complex manifold $(M, \mathcal{J})$ is called
	generalized Calabi-Yau if the bundle $U_\mathcal{J}$ is
	\emph{holomorphically trivial}. This is equivalent to the existence of a nowhere 
    vanishing global section $\rho \in C^\infty(U_{\mathcal{J}})$  (a non-vanishing pure spinor)	
	satisfying $d_H \rho=0$, where $d_H= d + H \wedge $ is the twisted de Rham differential. 
	\label{defn:gcy}
\end{defn}

\begin{defn}[{\cite[Def. 6.3]{gualtieri1}}] A \emph{generalized K\"ahler} structure is a commuting pair $(\mathcal{J}_1, \mathcal{J}_2)$ of generalized complex structures such that $G = - \mathcal{J}_1 \mathcal{J}_2$ is a positive definite metric on $E$. 
\label{defn:gk}
\end{defn}
\begin{ex}
Let $(g, J, \omega)$ be a usual K\"ahler manifold, then the following generalized complex structures:
\begin{equation}
\mathcal{J}_1 = \begin{pmatrix} -J & 0 \\ 0  &  J^* \end{pmatrix}, \qquad \mathcal{J}_2 = \begin{pmatrix} 0 &  \omega^{-1} \\ -\omega & 0 \end{pmatrix} 
\label{eq:3.1}
\end{equation}
commute and 
\begin{equation}
G = - \mathcal{J}_1 \mathcal{J}_2 = \begin{pmatrix} 0 & g^{-1} \\ g & 0 \end{pmatrix} 
\label{eq:3.2}
\end{equation} 
is a positive definite metric on $T \oplus T^*$.
\label{ex:gkusual}
\end{ex}
  The following notation is taken from \cite{gualtieri1}. Since $\mathcal{J}_1$ and $\mathcal{J}_2$ commute, we have the following decomposition
\begin{equation}
E  \otimes \mathbb{C} \cong (T \oplus T^*) \otimes \mathbb{C}= L_1^+ \oplus L_1^- \oplus \overline{L_1^+} \oplus \overline{L_1^-}~,
\label{eq:3.3}
\end{equation}
where $L_1 = L_1^+ \oplus L_1^-$ is the $+i$ eigenvalue bundle for  $\mathcal{J}_1$ and  
 $L_2 = L_1^+ \oplus \overline{L_1^-}$ is the $+i$ eigenvalue bundle for  $\mathcal{J}_2$.
 The Courant integrability of both $\mathcal{J}_1$ and $\mathcal{J}_2$ imply that each of the terms in the RHS of \eqref{eq:3.3} is Courant involutive.
If we define $C_\pm$ to be the $\pm 1$ eigenbundle of $G$, we obtain that 
\begin{equation}
C_\pm \otimes \mathbb{C} = L_1^\pm \oplus \overline{L_1^\pm},
\label{eq:3.4}
\end{equation}
Note that $C_+$ (resp $C_-$) is positive definite (resp. negative definite) with respect to the inner product on $E$.

\begin{defn}
A \emph{generalized Calabi-Yau metric} manifold is a generalized K\"ahler manifold $(M, \mathcal{J}_1, \mathcal{J}_2)$ such that both $(M, \mathcal{J}_1)$ and $(M, \mathcal{J}_2)$ are generalized Calabi-Yau with the corresponding $d_H$-closed pure spinors $\rho_1$ and $\rho_2$
  satisfying the following normalization condition 
\begin{equation}
( \rho_1, \overline{\rho_1}) = c (\rho_2, \overline{\rho_2}) 
\label{eq:3.5}
\end{equation}
for some constant $c$. 
\label{defn:gcym}
\end{defn}
 
\begin{ex}
Let $M$ be a usual Calabi-Yau manifold. We have the pure spinors
\begin{equation}
\rho_1 = \Omega,\qquad  \rho_2 = e^{i\omega}, 
\label{eq:3.6}
\end{equation}
where $\omega$ is the symplectic form and $\Omega$ is the holomorphic volume form. We have
\begin{equation}
(e^{i \omega}, e^{-i \omega}) = (-1)^{m (m-1)/2} (\Omega, \overline{\Omega}),
\label{eq:3.7}
\end{equation}
that is, $c = (-1)^{m (m-1)/2}$ where $m = \mathrm{dim} M$. 
\label{ex:3.2}
\end{ex}

\section{Sheaves of vertex algebras} \label{sec:3}
In this section we recall some results from \cite{gerbes2} and \cite{bressler1}
in the language of SUSY vertex algebras, following \cite{heluani8}.  In this section we do not require the Courant algebroid $E$ to be exact. The construction of the chiral-anti-chiral de Rham complex parallels that of the sheaf of (twisted) differential operators from a Lie algebroid (cf. Prop \ref{prop:universal} below).

Let $(E, \langle ,\rangle , [,], \pi)$ be a Courant algebroid. 
Let  $\Pi E$ be the corresponding purely odd super vector bundle. We will abuse
notation and denote by $\langle, \rangle$ the corresponding
super-skew-symmetric bilinear form, and by $[,]$ the corresponding  
degree $1$ bracket on $\Pi E$. Similarly, we obtain an odd
differential operator $\mathcal{D}: C^\infty(M) \rightarrow 
C^\infty(\Pi E)$. If no confusion should arise, when $v$ is an element of a vector
space $V$, we will denote by the same symbol $v$ the corresponding element of $\Pi
V$, where $\Pi$ is the \emph{parity change operator}.  

The following proposition from from \cite{heluani8} describes the construction of the chiral de Rham
 complex in parallel to the construction of \emph{twisted differential
 operators} given a Lie algebroid: 
 \begin{prop}
	 For each complex Courant algebroid $E$ over a differentiable manifold
	 $M$, there exists a sheaf $U^\mathrm{ch}(E)$ of SUSY vertex algebras on $M$
	 generated by functions $i: C(M) \hookrightarrow
	 U^\mathrm{ch}(E)$, and sections of $\Pi E$, $j: C( \Pi E)
	 \hookrightarrow U^\mathrm{ch}(E)$ subject to the relations:
	 \begin{enumerate}
		 \item $i$ is an ``embedding of algebras'', i.e. 
			 $i(1) = \vac$, and $i(fg) =
			 i(f) \cdot i(g)$, where in the RHS we use the normally
			 ordered product in $U^\mathrm{ch}(E)$.
		 \item $j$ imposes a compatibility condition between the
			 Dorfman bracket in $E$ and the Lambda bracket in
			 $U^\mathrm{ch}(E)$: \[ [j(A)_\Lambda j(B)] =
			 j ([A,  B]) + 2\chi i (\langle A, B \rangle).\]
		 \item $i$ and $j$ preserve the $\cO$-module structure of
			 $E$, i.e. $j (f A) = i(f)
			 \cdot j(A)$. 
		 \item $\mathcal D$ and $S$ are compatible, i.e. $j
			 \mathcal D f = S i (f)$. 
		 \item We impose the usual commutation relation \[
			 [j(A)_\Lambda i(f)] = i (\pi(A) f).\]
	 \end{enumerate}
	 In the particular case when $E = (T\oplus T^*) \otimes
	 \mathbb{C}$ is the standard
	 Courant algebroid with $H=0$, then $U^\mathrm{ch}(E)$ is the chiral-anti-chiral de Rham
	 complex of $M$ as in \cite{malikov}, denoted by $\Omega^\mathrm{ch}_M$ for historical
	 reasons\footnote{We call this sheaf chiral-anti-chiral as in \cite{frenkelnekrasov} not to confuse it with the \emph{holomorphic} chiral de Rham complex.}.
	 \label{prop:universal}
 \end{prop}
Using this proposition, we will abuse notation and use the same symbols for
sections of $E\otimes \mathbb{C}$ when they are viewed as sections of
$U^\mathrm{ch}(E)$.

\section{Bihermitian setup} \label{sec:4}

Following \cite{gualtieri1}, we discuss the bihermitian description of generalized K\"ahler geometry which 
we are going to use extensively later on. Let $(M, \mathcal{J}_1, \mathcal{J}_2)$ be a generalized K\"ahler manifold as in Definition \ref{defn:gk}.

The projection $\pi:E \cong T \oplus T^* \rightarrow T$ induces isomorphisms $\pi^\pm: C_\pm \xrightarrow{\sim} T$. We use these isomorphisms to transport structures from $C_\pm$ to $T$.  Restricting the natural symmetric and skew-symmetric pairings  on $T\oplus T^*$ to $C_\pm$ we obtain Riemannian metrics and two forms on both of $C_\pm$. We can transport these via $\pi^\pm$ to $T$ obtaining $b \pm g$, where $b$ is a two form, and $g$ is a Riemannian metric. Since $C_\pm$ are stable under both $\mathcal{J}_1$ and $\mathcal{J}_2$, we obtain complex structures on both of them which are compatible with the inner product. Projecting $\mathcal{J}_1$ with $\pi^\pm$ we obtain two Hermitian almost complex structures $J_\pm$ on $T$. Since $\mathcal{J}_1 = \pm \mathcal{J}_2$ in $C_\pm$ we would obtain the same data projecting $\mathcal{J}_2$. Finally let $\omega_\pm = g J_\pm$. 

We have constructed the data $(g, b, J_+, J_-)$ from $(\mathcal{J}_1, \mathcal{J}_2)$. It is easy to show that the latter can be recovered from the former as
\begin{equation}
\mathcal{J}_{1,2} = \frac{1}{2} \begin{pmatrix} 1 & 0 \\ b & 1 \end{pmatrix} \begin{pmatrix} J_+ \pm J_- & - \left( \omega_+^{-1} \mp \omega^{-1}_- \right) \\ \omega_+ \mp \omega_- & - \left( J^*_+ \pm J^*_- \right) \end{pmatrix} \begin{pmatrix} 1 & 0 \\ -b & 1 \end{pmatrix}.
\label{eq:b4.1}
\end{equation}

The projection $\pi$ identifies
\begin{equation}
\pi : L_1^\pm \xrightarrow{\sim} T^{1,0}_\pm~,
\label{eq:b4.2}
\end{equation}
Indeed writing explicitly $\pi$ we get
\begin{equation}
L_1^\pm = \Bigl\{ X + (b \mp i \omega_\pm) X | X \in C^\infty (T^{1,0}_\pm) \Bigr\}~.
\label{eq:b4.3}
\end{equation}
We can now write the integrability conditions for a generalized K\"ahler structure in terms of the bihermitian data. Here we review the relevant 
 results. 
\begin{prop}[{\cite{gualtieri1}}] The complex structures $J_\pm$ coming from a generalized K\"ahler structure are integrable and the forms $\omega_\pm$, $b$ and $H$ satisfy
\begin{equation}
 d^c_\pm \omega_\pm  = \pm (db + H)
\label{eq:b4.4}
\end{equation}
where $d^c_\pm = i (\overline{\partial}_\pm - \partial_\pm)$ and $\partial_\pm$ is the $\partial$ operator for the complex structure $J_\pm$.
\label{prop:b4.1}
\end{prop}
\begin{prop}[{\cite{gualtieri1}}] Let $(g,b, J_\pm)$ be the bihermitian data obtained from a generalized K\"ahler manifold. Define two connections 
\begin{equation}
\nabla^\pm = \nabla \pm \frac{1}{2} g^{-1}(db + H) ~,
\label{eq:b4.5}
\end{equation}
where $\nabla$ is the Levi-Civita connection for $g$. We obtain $\nabla^\pm J_\pm = 0$ and $(db+H)$ is of type $(2,1) + (1,2)$ with respect to both $J_\pm$. 
\label{prop:4.2}
\end{prop}
\begin{rem}
  As far as the bihermitian picture is concerned the only data we use is the combination $db+H$ which gives rise to a closed 3 form.
   In view of this statement, we may replace $H$ by $H + db$ in the definition of the Courant algebroid, hence we may assume that $b = 0$ in all formulas. 
   Thus without loss of any generality we can set $b$ to zero in all above propositions and use only $H$.  This is what we do in the rest of 
    paper. 
\end{rem}
\begin{rem}
The generalized complex structures ${\mathcal J}_{1,2}$ give rise to the following Poisson tensors 
\begin{equation}
 P_{1,2} = - \omega_+^{-1} \pm \omega_-^{-1}~.
\label{definitionPoisson}
\end{equation}
\end{rem}

Below we will need some properties of the Bismut connection $\nabla^\pm$. For each complex structure $J_\pm$ we choose a holomorphic system of coordinates $\{z^\alpha_\pm\}$. We will use Greek subindexes when using these coordinate systems while 
we use Latin subindexes for a general coordinate system. Since the Hermitian complex structures are covariantly constant
\begin{equation}
\nabla^\pm J_\pm = 0  \Rightarrow {{J_\pm}^i_j}_{,k}= \Gamma^{\pm l}_{kj} {J_\pm}^i_{l} - \Gamma^{\pm i}_{kl} {J_\pm}^l_{j}~,
\label{eq:6.9}
\end{equation}
where 
\begin{equation}
  \Gamma^{\pm l}_{kj} =  \Gamma^{\pm l}_{kj} \pm g^{ls} H_{skj}
\label{eq:6.933uu3}
\end{equation}
 where $  \Gamma^{\pm l}_{kj}$ are the Christoffel symbols of the Levi-Civita connection for $g$. 
In the coordinate system $\{z_\pm^\alpha\}$ these imply
\begin{equation}
\Gamma^{\pm \alpha}_{i \bar \beta} = \Gamma^{\pm \bar \alpha}_{i \beta} = 0 \qquad \forall \alpha, \beta, i,
\label{eq:6.10}
\end{equation}
and from these we infer:
\begin{equation}
\Gamma^{\pm \alpha}_{\bar \beta \alpha} = \pm H_{\bar \beta \alpha \bar \gamma} g^{\alpha \bar \gamma}, \qquad \Gamma^{\pm \bar \alpha}_{\beta \bar \alpha} = \pm H_{\beta \bar \alpha \gamma} g^{\gamma \bar \alpha}~.
\label{eq:6.11}
\end{equation}
We also define the following one forms
\begin{multline}
v^\pm_i = J^j_{\pm i} \nabla_k J^k_{\pm j}= J^j_{\pm i} \nabla_k^\pm J_{\pm j}^k \mp \frac{1}{2} J^j_{\pm i} H_{klm}g^{km} J^l_{\pm j} \pm \frac{1}{2}  J^j_{\pm i} H_{kjm} g^{ml} J_{\pm l}^k \\ = \pm \frac{1}{2} J^j_{\pm i} J^k_{\pm l} g^{ml} H_{kjm}~,
\label{eq:6.11nose}
\end{multline}
which in the holomorphic coordinate system $\{z^\alpha_\pm\}$ looks like
\begin{equation}
2 v^\pm_\alpha = \mp H_{\beta \alpha \bar \gamma} g^{\beta \bar \gamma} \pm H_{\bar \beta \alpha \gamma} g^{\gamma \bar \beta} = \mp 2 H_{\beta \alpha \bar \gamma} g^{\beta \bar \gamma} = - 2 \Gamma_{\alpha \bar \beta}^{\pm \bar \beta}~. 
\label{eq:6.11dilaton1}
\end{equation}
Similarly we obtain 
\begin{equation}
2 v^\pm_{\bar \alpha} = \mp H_{\bar \beta \bar \alpha \gamma} g^{\gamma \bar \beta} \pm H_{\beta \bar \alpha \bar \gamma} g^{\beta \bar \gamma} = - 2 \Gamma^{\pm \beta}_{\bar \alpha \beta}~.
\label{eq:6.11dilaton2}
\end{equation}
 It is convenient to introduce local frames on $E$ adapted to the decomposition \eqref{eq:3.3}. 
According to \eqref{eq:b4.3} we can choose local frames for $L_1^\pm$ given by
\begin{equation}
e_\alpha^\pm = \frac{\partial}{\partial z_\alpha^\pm} \pm  g_{\alpha \bar \beta} d z^\pm_{\bar\beta}
\label{eq:b4.7}
\end{equation}
with dual frames on $\overline{L_1^\pm}$
\begin{equation}
e^\alpha_{\pm} = \pm g^{\alpha \bar\beta} e^\pm_{\bar \beta} := \pm g^{\alpha \bar\beta} \left( \frac{\partial}{\partial z^\pm_{\bar \beta}} \pm g_{\gamma \bar{\beta}} dz^\pm_\gamma \right) = \left( dz^\pm_\alpha \pm g^{\alpha \bar \beta}  \frac{\partial}{\partial z^\pm_{\bar \beta}} \right)~.
\label{eq:b4.8}
\end{equation}
With respect to complex conjugation we have the following properties
\begin{equation}
    e^\pm_{\bar \alpha} = \overline{e^\pm_\alpha} = \pm g_{\bar{\alpha}\beta} e^\beta_\pm~,
   \label{complexconj}
   \end{equation}
  where these  expressions are written in the holomorphic coordinate system for $J_\pm$ correspondingly.  One can easily calculate some of the  Dorfman 
    brackets
    \begin{equation}
     [e^\pm_\alpha, e^\pm_\beta]=0~, ~~~~~~~   [e_\pm^\alpha, e_\pm^\beta]=0~, 
    \label{dorfgai}
    \end{equation} 
     while other brackets are non-trivial. 

\section{Unimodularity as twisted Ricci-flatness} \label{sec:5}

In this section we prove some useful identities about the divergences of the local frame elements $\{e_\alpha^\pm\}$ and the corresponding statements in terms of the bihermitian data. First we recall some basic notions from the theory of Lie algebroids and generalized complex geometry.

 Recall that given any Lie
algebroid $L$ over $M$, we can define a differential $d_L: C^{\infty}(M)
\rightarrow C^{\infty}(L^*)$ as $(d_L f)(l) = \pi_L (l) f$, where $l$ is
a section of $L$ and $\pi_L$ is the anchor map of $L$. This differential
can be extended to $\bigwedge^\bullet L^*$ by imposing the Leibniz rule
in the usual way (for $\zeta \in C^{\infty} (\bigwedge^{k-1} L^*)$):
\begin{multline}
	(d_L \zeta)(l_1, \dots, l_{k}) = \sum_i (-1)^{i+1} \pi(l_i)
	\zeta(l_1, \dots, \hat{l_i}, \dots l_k) + \\\sum_{i<j} (-1)^{i+j}
	\zeta([l_i, l_j], \dots, \hat{l_i}, \dots, \hat{l_j}, \dots,
	l_k).
	\label{eq:diffeerentialmult}
\end{multline}
The cohomologies of the complex $(\bigwedge^\bullet L^*, d_L)$ are
denoted by $H^\bullet (L)$ and are called the Lie algebroid cohomologies
of $L$ (with trivial coefficients).

Now let $(M, \mathcal{J})$ be a generalized complex manifold with trivial $U_{\mathcal{J}}$. Given a non-vanishing global
section of $U_\mathcal{J}$, we obtain an isomorphism of sheaves:
\begin{equation}
	\wedge^k \overline{L} \simeq U_{k-n}~,
	\label{eq:isomalgebroid}
\end{equation}
where $U_{k-n}$ where defined in \eqref{eq:ukn}. 
The twisted  de Rham differential can be
split as $d_H = \partial + \bar\partial$ such that $\partial: U_k
\rightarrow U_{k-1}$ and $\bar\partial :U_k \rightarrow U_{k+1}$. Suppose moreover that $M$ is generalized Calabi-Yau, in this case the isomorphism (\ref{eq:isomalgebroid}) allows us to
identify the complex $(U_\bullet, \bar{\partial})$ with the complex
computing the Lie algebroid cohomology of $L$
(using $\bar{L} = L^*$).
Moreover, in this case, the Lie algebroids $L$ and $L^*$ are both
\emph{unimodular}, a notion due to Weinstein \cite{weinstein2} that we now recall. 

For a Lie algebroid $L$ we have the corresponding sheaf of twisted
differential operators $U(L)$. The sheaf $\det T^*$ is
always a right twisted D-module, and the corresponding left $U(L)$-module
is then the line bundle $Q_L = \det L \otimes
\det T^*$. Suppose for simplicity that the line bundle $Q_L$ is trivial,
for each non-vanishing section $s$ of $Q_L$ we can define $\theta_s \in
C^{\infty} (L^*)$ by \[ \theta_s(l) s = l \cdot s, \] where we use the left
D-module structure of $Q_L$ on the RHS.  It turns out that $\theta$ gives
rise to a well 
defined element of $H^1 (L, Q_L)$, the first Lie algebroid
cohomology of $L$ with coefficients in $Q_L$ (see \cite{weinstein1} for
details).
\begin{defn}[\cite{weinstein2}]
	A Lie algebroid $L$ is called \emph{unimodular} if the class
	$\theta \in H^1(L, Q_L)$ above constructed vanishes.
	\label{defn:unimodularity}
\end{defn}
We have the following Proposition (see for example \cite[Theorem
10]{zabzine4})
\begin{prop}
	A generalized complex manifold $M$ is generalized Calabi-Yau if
	and only if $U_\mathcal{J}$ is trivial and $L$ is unimodular.
	\label{prop:unimodularity}
\end{prop}
In fact this can be refined as follows. Let $(M, \mathcal{J})$ be a generalized complex manifold with topologically trivial canonical bundle. Let $\rho$ be a non-vanishing section of $U_{\mathcal{J}}$. Integrability of $\mathcal{J}$ implies that there exists a unique $\chi \in C^\infty(L^*)$ such that
\begin{equation}
d_H \rho = \chi \cdot \rho.
\label{eq:d.0.1}
\end{equation}
There also exists a unique section $\zeta \in C^\infty(\det L^*)$ such that 
\begin{equation}
\bar{\rho} = \zeta \cdot \rho,  
\label{eq:alamiercoles}
\end{equation}
Recall that the spinor $\rho$ gives rise to a volume form 
\begin{equation}
\mu :=  (\rho, \bar{\rho}) = [\rho^\top \wedge \zeta \cdot \rho]_{\mathrm{top}} \in C^\infty (\det T^*)~,
\label{eq:d.0.3}
\end{equation}
 where $(~,~)$ is Mukai pairing. Therefore we can define the section $s$ of $\det L \otimes \det T^*$ by $s = \bar{\zeta} \otimes \mu$. 
We have the following
\begin{prop}
The modular class $\theta_s$ is represented by $2\chi$. 
\label{prop:ahad}
\end{prop}
\begin{proof}
Recall we have an isomorphism $\overline{L} \simeq L^*$. This induces an isomorphism $\det \overline{L} \simeq \det L^*$. In the basis $\wedge \overline{e_{i}}$ and $\wedge e^i$ this isomorphism is given by multiplication by a function $\alpha$\footnote{We can always make this function equal $1$ but we are interested in the frames \eqref{eq:b4.7} and \eqref{eq:b4.8} adapted to the bihermitian structure, in which case $\alpha = \log \det g$ (cf. Remark \ref{rem:normalization}) .}. It follows from \eqref{eq:alamiercoles} and its complex conjugate that $\zeta \bar{\zeta} = 1$. 
Let $\{e_i\}$, be a local frame for $L$ with dual frame $\{e^i\}$, we can locally write 
\begin{equation}
\zeta = e^{-i\psi} \sqrt{\alpha} \,  e^1 \wedge \dots \wedge e^{\dim M} \in C^\infty (\det L^*).
\label{eq:d.0.2}
\end{equation}
for a real function $\psi$. And we have 
\begin{equation}
\bar \zeta = \frac{e^{i \psi}}{\sqrt{\alpha}}\, e_1 \wedge \dots \wedge e_{\dim M} \in C^\infty (\det L). 
\label{eq:ahalohice}
\end{equation}
Since $s = \bar{\zeta} \otimes \mu = \bar{\zeta} \otimes (\rho, \zeta \cdot \rho)$ we may assume $\alpha = 1$ and $\psi = 0$.
\begin{equation}
\dive_\mu (e_i) \cdot \mu := - \lie_{\pi e_i} \mu = - d \iota_i ( \rho, \zeta \cdot \rho)
\label{eq:d.2}
\end{equation}
where $\iota_i \, \cdot := \pi(e_i) \, \cdot$. 
Let us write $\rho = \sum_p \rho_p$ where $\rho_p \in C^{\infty} (\wedge^p T^*)$. Then we have:
\begin{equation}
\iota_i \left [ \rho^\top \wedge \zeta \cdot \rho \right ] = \sum_p (-1)^{p-1} \left [ \iota_i \rho_p \right ]^\top \wedge \zeta \cdot \rho + \sum_p (-1)^p \rho_p^\top \wedge \iota_i \zeta \cdot \rho, 
\label{eq:d.3}
\end{equation}
Now using that $e_i \cdot \rho = 0$ we have
\begin{equation}
\begin{split}
\iota_i \left [ \rho^\top \wedge \zeta \cdot \rho \right ] &= - \sum_p (-1)^{p-1} \left [\pi^*(e_i) \wedge \rho_p \right ]^\top \wedge \zeta \cdot \rho + \sum_p (-1)^p \rho_p^\top \wedge \iota_i \zeta \cdot \rho, \\
 &= \sum_p (-1)^p \rho_p^\top \wedge \pi^*(e_i) \wedge \zeta \cdot \rho + \sum_p (-1)^p \rho_p^\top \wedge \iota_i \zeta \cdot \rho, \\
 &= \sum_p (-1)^p \rho_p^\top \wedge e_i \cdot \zeta \cdot \rho.
\end{split}
\label{eq:d.4}
\end{equation}
Taking $d_H$ of this we obtain:
\begin{multline}
\dive_\mu (e_i) \mu = - \bigl[(d_H \rho)^\top \wedge e_i \cdot \zeta \cdot \rho \bigr]_{\mathrm{top}} - \bigl[ \rho^\top \wedge d_H e_i \cdot \zeta \cdot \rho \bigr]_{\mathrm{top}} \\ = - \bigl[ (\chi \cdot \rho)^\top \wedge e_i \cdot \zeta \cdot \rho \bigr]_{\mathrm{top}} - \bigl[\rho^\top \wedge d_H e_i \cdot \zeta \cdot \rho \bigr]_{\mathrm{top}} = - ( e_i \cdot \chi \cdot \rho, \zeta \cdot \rho ) - ( \rho, d_H e_i \cdot \zeta \cdot \rho  )\\ = - \chi(e_i) \mu - ( \rho, d_H e_i \cdot \zeta \cdot \rho )  = - \chi (e_i) \mu - ( \rho, [d_H, e_i] \cdot \zeta \cdot \rho ) \\ = - \chi(e_i) \mu - ( \rho, [ [ d_H, e_i], \zeta] \cdot \rho ) - ( \rho, \zeta \cdot e_i \cdot \chi \rho  ) = - 2 \chi(e_i)\mu  - ( \rho, [ [d_H, e_i], \zeta] \cdot \rho ) .
\label{eq:d.6}
\end{multline}
Where for any two elements $a,b$ of the Clifford algebra of $(E, \langle \cdot, \cdot\rangle)$ we write $[a,b] = a\cdot b - (-1)^{p(a)p(b)} b \cdot a$. Using the fact that the Dorfman bracket is defined as a derived bracket \cite{gualtieri2} we obtain:
\begin{equation}
\dive_\mu(e_i) \mu= - 2 \chi(e_i) \mu - ( \rho, [e_i, \zeta]  \cdot \rho ).
\label{eq:d.7}
\end{equation}
Since we have the equation
\begin{equation}
[e_i, e^j] = c_i^{jk} e_k - c_{ik}^j e^k,
\label{eq:d.8}
\end{equation}
for some functions $c_{i}^{jk}$ and $c_{ik}^j$, we obtain 
\begin{equation}
\dive_\mu (e_i) = - 2 \chi (e_i) + c_{ij}^j 
\label{eq:d.9}
\end{equation}
On the other hand, by definition of the modular class $\theta_s$ of the Lie algebroid $L$ we have
\begin{equation}
\theta_s(e_i) = c_{ij}^j - \dive_\mu (e_i).
\label{eq:d.11}
\end{equation}
From where we obtain that the modular class is represented by $2 \chi$. 
\end{proof}
We will also need the following  
\begin{prop}
 On a generalized Calabi-Yau manifold with closed pure spinor $\rho$ and corresponding 
  volume form
 $(\rho, \bar{\rho}) = \mu$,
  the divergence of the corresponding Poisson structure $P$ with respect to $\mu$ vanishes:  $\dive_\mu P=0$. 
\label{prop:poismod}
\end{prop}
\begin{proof}
 This proposition is a simple corollary of \cite[Prop. 3.27]{gualtieri2}. The Poisson structure, pure spinor and volume form are related by 
  \begin{equation}
  \rho^\top \wedge \bar{\rho}  = e^{-\frac{iP}{2}} \mu~.  
 \label{ddjj33}
 \end{equation}
  Since $\rho$ is a pure spinor we have $ d (e^{-\frac{iP}{2}} \mu)=0$ from where the proposition follows. We 
 remind the reader that the divergence of a multivector is defined as follows    \begin{equation}
      \dive_\mu P \cdot \mu =  d (P \cdot \mu)~,
   \label{2292jsjs}
   \end{equation}
    where by $P\cdot \mu$ we understand the contraction of the  multivector $P$ with the form $\mu$. Thus  in local coordinates the divergence of 
  the Poisson structure $P$ can be written as
  \begin{equation}
   (\dive_\mu P )^j =  \frac{1}{\tilde{\mu}} \partial_i (\tilde{\mu}\,P^{ij}) = \partial_i P^{ij} + \partial_i (\log \tilde{\mu})\, P^{ij}~,
  \label{wiiiwiw3}
  \end{equation}
 where the volume form $\mu = \tilde{\mu}~ dx^1 \wedge ... \wedge dx^{2n}$. 
\end{proof}
In the rest of this section we fix a generalized Calabi-Yau metric manifold $(M, \mathcal{J}_1, \mathcal{J}_2)$ with its two pure spinnors $\rho_1$ and $\rho_2$. Recall that we have the decomposition \eqref{eq:3.3} and we choose frames $\{e_\alpha^\pm\}$ for $L_1^\pm$ with dual frames $\{e^\alpha_\pm\}$. 
\begin{rem}
Using the frames \eqref{eq:b4.7} and \eqref{eq:b4.8} we can explicitly identify the functions $\alpha$ in \eqref{eq:d.0.2}. We obtain 
\begin{equation}
\begin{aligned}
\rho_1 &=\frac{1}{\sqrt{\det g}} e^{i\psi_1} e^+_1 \wedge \dots e^+_n\wedge e^-_1 \wedge \dots \wedge e^-_n \cdot \bar{\rho}_1~, \\  \rho_2 &= e^{i\psi_2} e^+_1 \wedge \dots \wedge e^+_n \wedge e^1_- \wedge \dots \wedge e^n_- \cdot \bar{\rho}_2~,
\end{aligned}
\label{eq:6.1ddjdjjd}
\end{equation}
for two real functions $\psi_1$ and $\psi_2$. 
\label{rem:normalization}
\end{rem}
The pure spinors give rise to the volume forms:
 \begin{equation}
   ( \rho_1, \overline{\rho_1}) = c (\rho_2, \overline{\rho_2})  = e^{-2\Phi} \vol_g~,
 \label{definitiondil}
 \end{equation}
 where $\mathrm{vol}_g$ is  the volume form induced by the Riemannian metric $g$  which can be written in any coordinate system as
\begin{equation}
\vol_g = \sqrt{\det g}~ dx^1 \wedge \dots dx^{2n},
\label{eq:4honom2}
\end{equation}
 and $\Phi$ is a function showing the mismatch between the Riemannian volume form and the volume form 
  induced by the pure spinors. In the physics literature such function $\Phi$ is called a \emph{dilaton} and equation \eqref{definitiondil}
   should be regarded as the definition of the dilaton. 

  Proposition \ref{prop:poismod} implies that the divergences of the Poisson structures 
  $P_1$ and $P_2$ in \eqref{definitionPoisson} calculated with 
  respect to $e^{-2\Phi}\vol_g$ are zero. Therefore the divergences of $\omega_\pm^{-1}$ are zero:
  \begin{equation}
   \frac{1}{e^{-2\Phi} \sqrt{ \det g}} \partial_i \left ( e^{-2\Phi} \sqrt{\det g}~ \omega_\pm^{ij} \right )=0.
  \label{ejej333asskd}
  \end{equation}
   This in turn implies
 \begin{equation}
v^\pm = - 2 d \Phi~,
\label{eq:6.11dilaton3}
\end{equation}
 where the $1$-forms $v^\pm$ are defined in \eqref{eq:6.11nose}.  The relations \eqref{eq:6.11dilaton1} and \eqref{eq:6.11dilaton2} imply 
\begin{equation}
 2 \frac{\partial}{\partial z^\alpha_{\pm}} \Phi = \Gamma_{\alpha \bar \beta}^{\pm \bar\beta}~, \quad \quad
 2 \frac{\partial}{\partial z^{\bar \alpha}_{\pm}} \Phi =  \Gamma^{\pm \beta}_{\bar \alpha \beta}~.
 \label{sjsj333}
\end{equation}

Let us analyze unimodularity in this context. Recall that the manifolds $(M, \mathcal{J}_1)$ and $(M, \mathcal{J}_2)$ are generalized Calabi-Yau therefore $L_1 = L_1^+ \oplus L_1^-$ and $L_2 = L_1^+ \oplus \overline{L_1^-}$ are both unimodular. Not only the modular classes vanish but Proposition \ref{prop:ahad} says that they are represented by zero (as opposed to an exact form) when using the appropriate volume form. Define the structure functions
\begin{equation}
\begin{aligned}
c_{\alpha \beta}^{\pm\gamma} &= \langle [e^\pm_\alpha, e^\pm_\beta], e_\pm^\gamma \rangle, & d_{\alpha \beta}^\gamma &= \langle [e^+_\alpha, e^-_\beta], e_+^\gamma \rangle, & e_{\alpha \beta}^\gamma &= \langle [e^+_\alpha, e^-_\beta], e_-^\gamma \rangle, \\
c^{\alpha \beta}_{\pm \gamma} &= \langle [e_\pm^\alpha, e_\pm^\beta], e^\pm_\gamma\rangle, & d^{\alpha \beta}_\gamma &= \langle [e_+^\alpha, e_-^\beta], e^+_\gamma\rangle, & e^{\alpha \beta}_\gamma &= \langle [e_+^\lambda, e_-^\beta], e^-_\gamma \rangle,
\end{aligned}
\label{eq:6.5}
\end{equation}
 where  $c_{\alpha \beta}^{\pm\gamma}$ and $c^{\alpha \beta}_{\pm \gamma}$ vanish if we use the coordinate frames \eqref{eq:b4.7} and \eqref{eq:b4.8}. Using \eqref{eq:6.1ddjdjjd} we 
 compute explicitly the representatives of the modular classes for $L_1$ and $L_2$ to obtain:
\begin{equation}
\begin{aligned}
\theta_{1}(e_\alpha^+) &=    c_{\alpha \beta}^{+\beta} + e_{\alpha\beta}^\beta -  \dive e^+_\alpha + \pi (e_\alpha^+) \cdot (- \log \sqrt{\det g} + i \psi_1) = 0~,\\
\theta_{2}(e_\alpha^+) &=     c_{\alpha\beta}^{+\beta} - e_{\alpha \beta}^\beta - \dive e^+_\alpha + \pi(e_\alpha^+) \cdot (i\psi_2) = 0 ~,
\end{aligned}
\label{eq:6.6}
\end{equation}
 where the divergences are calculated with respect to $\mu=e^{-2\Phi} \vol_g$. 
  Taking the sum of the above expressions and using  the coordinate frames \eqref{eq:b4.7} and \eqref{eq:b4.8} we obtain:
   \begin{equation}
    \frac{\partial}{\partial z^\alpha_+} \left  (  \log (e^{-2\Phi} (\det g)^{1/4} )+ \frac{i}{2} (\psi_1 + \psi_2)  \right ) =0~. 
   \label{2822sjs}
 \end{equation}
  Taking the derivative with respect $z_+^{\bar{\alpha}}$ of this equation and retaining the real part we arrive to:
    \begin{equation}
    \frac{\partial}{\partial z^\alpha_+}   \frac{\partial}{\partial z^{\bar{\beta}}_+} \left  (  \log (e^{-4\Phi} \sqrt{\det g} )  \right ) =0~. 
   \label{2822sjkkds}
 \end{equation}
  We can derive the analogous statement for the complex structure $J_-$  by evaluating $\theta_1(e^-_\alpha)$ and $\theta_2(e^-_\alpha)$. 
   Equation \eqref{2822sjkkds} implies that 
   \begin{equation}
    \log \sqrt{\det g} = \varphi^\pm + \bar{\varphi}^\pm + 4 \Phi~,
   \label{eq:6.relation.g}
   \end{equation}
    where $\varphi^\pm$ is holomorphic with respect to the complex structure $J_\pm$ respectively. In other words it implies that 
  there exist  non-vanishing  holomorpic volume forms. In local complex  coordinates we write them as
\begin{equation}
\Omega_\pm = e^{\varphi^\pm} dz^1_\pm \wedge \dots dz^n_\pm,
\label{eq:4holom1}
\end{equation}
 such that  
 $$\Omega_\pm \wedge \overline{\Omega_\pm} = e^{-4\Phi} \vol_g~.$$
  We can rescale these nowhere vanishing  forms
\begin{equation}
\zeta_\pm = e^{2 \Phi + \varphi^\pm} dz_\pm^1 \wedge \dots dz_\pm^n~,
\label{eq:4holom4}
\end{equation}
 such that $\zeta_\pm \wedge \bar{\zeta}_\pm = \vol_g$.  Since $g$ is parallel with respect to both $\nabla^\pm$
 we have:
\begin{equation}
 \partial_{z^\alpha_\pm} ( \log\sqrt{\det g} )
  = \Gamma^{\pm \beta}_{\alpha \beta} + \Gamma^{\pm \bar \beta}_{\alpha \bar \beta},  \quad 
   \partial_{z^{\bar{\alpha}}_\pm} ( \log\sqrt{\det g}) = \Gamma^{\pm \bar{\beta}}_{\bar\alpha \bar \beta} + \Gamma^{\pm \beta}_{\bar \alpha \beta}~.
\label{eq:6.11c}
\end{equation}
 Using \eqref{sjsj333}  we find the traces of the Christoffel symbols as:
\begin{equation}
\Gamma^{\pm \beta}_{\alpha \beta} = \partial_{z^\alpha_\pm} \left( \varphi^\pm + 2 \Phi \right)~ , \qquad \Gamma^{\pm \bar\beta}_{\bar\alpha \bar\beta} = \partial_{z^{\bar\alpha}_\pm} \left( \overline{\varphi^\pm} + 2 \Phi \right)  ~.
\label{eq:6.11g}
\end{equation}
 This implies that the forms $\zeta_\pm$ are covariantly constant:  $\nabla^\pm \zeta_\pm =0$.  
  Thus the generalized Calabi-Yau metric manifold has $SU(n)$ holonomy for $\nabla^\pm$. 
 
Finally let us calculate some of the traces of structure functions which we will need to use later. 
Using the explicit  frames \eqref{eq:b4.7} and \eqref{eq:b4.8} a direct computation shows
\begin{equation}
[e^\pm_\alpha, e_\pm^\alpha] = \Gamma^{\pm \alpha}_{\beta \alpha} \left(dz^\beta_{\pm} \mp g^{\beta \bar \gamma} \partial_{z^{\bar \gamma}_{\pm}} \right) + 2 \Gamma^{\pm \alpha}_{\bar \beta \alpha} dz^{\bar \beta}~.
\label{eq:6.12}
\end{equation}
Using \eqref{sjsj333} and \eqref{eq:6.11g} we can rewrite this as
\begin{equation}
[e^\pm_{\alpha}, e^\alpha_{\pm}] = d \varphi^\pm \mp g^{-1} d\varphi^\pm - 2 ( \partial^\pm \Phi \pm   g^{-1} \partial^\pm \Phi) + 4 d \Phi~,
\label{eq:6nuevisima}
\end{equation}
 where $d$ is de Rham differential and $d= \partial^\pm + \overline{\partial}^\pm$ is the decomposition with respect to the complex structures $J_\pm$ respectively. 
We are interested in the inner product of this expression with $e_\beta^\mp$ and $e^\beta_\mp$. 
 Using the orthogonality of frames  we obtain a coordinate independent expression for \eqref{eq:6nuevisima}:
\begin{equation}
\langle[e_\alpha^\pm, e^\alpha_\pm], e^\mp_\beta \rangle = \langle d\varphi^\pm \mp g^{-1} d\varphi^\pm + 4 d \Phi, e^\mp_\beta \rangle.
\label{eq:nueva3}
\end{equation}
This can easily be evaluated to obtain:
\begin{equation}
\begin{aligned}
 d_{\alpha \beta}^\alpha &= - \partial_{z^\beta_-} \left( \varphi^+  + 2 \Phi \right) = - \pi(e^-_\beta) \left( \varphi^+  + 2 \Phi \right) , \\ e_{\beta \alpha}^\alpha &=  \partial_{z^\beta_+} \left( \varphi^- + 2 \Phi \right) =  \pi(e^+_\beta) \left( \varphi^- + 2 \Phi \right)~.
 \end{aligned}
\label{eq:6.13}
\end{equation}
Similarly, taking the inner product with $e^\beta_\mp$ we obtain
\begin{equation}
\begin{aligned}
d^{\alpha \beta}_\alpha &=  - g^{\beta \bar \gamma}_- \partial_{z^{\bar\gamma}_-} \left(\varphi^+ + 2 \Phi \right) =   \pi(e^\beta_-) \left( \varphi^+ + 2 \Phi \right), \\  e^{\beta \alpha}_\alpha &=  -g_+^{\beta \bar \gamma} \partial_{z^{\bar \gamma}_+} \left( \varphi^- + 2 \Phi \right) =  - \pi(e^\beta_+) \left( \varphi^- + 2 \Phi \right) ~.
\end{aligned}
\label{eq:6.14}
\end{equation}
In the next section we adopt the following short-hand notation for the action of the anchor $\pi$
\begin{equation}
f^{,\alpha^\pm}  = \pi(e_\pm^\alpha) f~,\quad  f_{,\alpha^\pm} = \pi(e^\pm_\alpha) f~. 
\label{notau229}
\end{equation}
Finally, we can now identify the functions $\psi_{1/2}$ of \eqref{eq:6.1ddjdjjd} obtained in the generalized Calabi-Yau metric context with their counterparts $\varphi^\pm$ obtained in the bihermitian setup. Using the frames in \eqref{eq:b4.7}--\eqref{eq:b4.8} and  taking the difference of the equations in \eqref{eq:6.6} we find:
\begin{equation}
2e_{\alpha \beta}^\beta + \partial_{z^\alpha_+} \bigl(i \psi_1 - i \psi_2  - \log \sqrt{\det g} \bigr) = 0.
\label{eq:6.5.1}
\end{equation}
Comparing with the second equation of \eqref{eq:6.13} we obtain 
\begin{equation}
\partial_{z^\alpha_+} \bigl( i \psi_1 - i \psi_2 - \log \sqrt{\det g} \bigr) = - 2 \partial_{z^\alpha_+} \bigl(\varphi^- + 2 \Phi \bigr).
\label{eq:6.5.2}
\end{equation}
Similarly, computing the modular classes of $L_{1,2}^*$ valuated in $e^\alpha_+$ we obtain
\begin{equation}
g^{\alpha \bar \beta} \partial_{z^{\bar\beta}_+} \bigl(i \psi_1 - i \psi_2 - \log \det \sqrt{g} \bigr) = - 2 g^{\alpha \bar \beta} \partial_{z^{\bar \beta}_+} \bigl(\varphi^- + 2 \Phi \bigr)
\label{eq:6.5.3}
\end{equation}
From where we may assume
\begin{equation}
i \psi_1 - i \psi_2  - \log \det \sqrt{g} = -2 \bigl( \varphi^- + 2 \Phi \bigr).
\label{eq:6.5.4}
\end{equation}
Similarly, by valuating the modular classes in $e_\alpha^-$ and $e^\alpha_-$ we obtain:
\begin{equation}
\log \sqrt{\det g} - i \psi_1 - i \psi_2 = 2 \bigl( \varphi^+ + 2 \Phi \bigr)
\label{eq:6.5.5}
\end{equation}
and comparing with \eqref{eq:6.5.4} we obtain:
\begin{equation}
\psi_1 = i \left( \varphi^\pm - \overline{\varphi^\mp} \right), \qquad \psi_2 = 
i \left( \varphi^+ - \varphi^- \right).
\label{eq:6.5.6}
\end{equation}
We finish this section by summarizing the relations between all the global sections found so far. On the bihermitian setup we have the holomorphic volume forms $\Omega^\pm$ defined in \eqref{eq:4holom1} and the corresponding covariantly constant forms $\zeta^\pm$ defined in \eqref{eq:4holom4}, these latter are global sections of $\det T^{*1,0}_\pm$. On the generalized Calabi-Yau metric setup we have the pure spinors $\rho_{1,2}$ and the sections $\zeta_{1,2}$ of $\det{L_{1,2}}^*$ defined by \eqref{eq:alamiercoles}. We see that under the dual of the isomorphism \eqref{eq:b4.2} the sections $\zeta^\pm$ are mapped to sections of $\det L_1^{\pm*}$, we can write them locally as 
\begin{equation}
\zeta^\pm = e^{\eta^\pm} e_\pm^1 \wedge \dots e_\pm^n, \qquad 2n = \dim M.
\label{eq:4.3}
\end{equation}
It follows from \eqref{eq:6.5.6}, \eqref{eq:6.1ddjdjjd} and \eqref{eq:6.relation.g} that the sections $\zeta_{1,2}$ are given by $e^{\eta^+ + \eta^-}$ and $e^{\eta^+ - \eta^-}$ respectively. 

\section{$N=2,2$ superconformal structure} \label{sec:6}

In this section we state and proof the main Theorem of this article. We start by associating to each pure spinor $\rho_{1,2}$ of  $(M, \mathcal{J}_{1,2})$ a global section of $U^\mathrm{ch}(E)$. We find the explicit description of these sections in terms of the covariantly constant forms \eqref{eq:4holom4} in the bihermitian setup. 

Let $(M, \mathcal{J})$ be a generalized complex manifold and let $\rho$ be a pure spinor. Let $L$ be the corresponding Lie algebroid (the $i$-eigenbundle of $\mathcal{J}$). Recall that $\rho$ gives rise to a global section  $\zeta \in C^\infty(\det L^*)$ given by \eqref{eq:alamiercoles}. Let $\{e_i\}$ be a local frame for $L$ and $\{e^i\}$ be the dual frame. We can write $\zeta$ locally as 
\begin{equation}
\zeta = e^\eta e^1 \wedge \dots e^n, \qquad n = \mathrm{dim} M.
\label{eq:4.1}
\end{equation}
The following Lemma is proved as in \cite[Lem 5.1]{heluani9}
\begin{lem}
The local section 
\begin{equation}
J = \frac{i}{2} \sum_i e^i e_i + i T \eta~,
\label{eq:4.2}
\end{equation}
gives a well defined global section of $U^\mathrm{ch}(E)$.
\label{lem:4.1}
\end{lem}

\begin{rem}
Integrability of $L$ is not used in the proof of this Lemma, so we may replace $\mathcal{J}$ by a generalized almost complex structure such that $\det L^*$ is trivial. The section $\zeta$ can be any non-vanishing section of $\det L^*$. 
\label{rem:4.1}
\end{rem}

Let now $(M, \mathcal{J}_{1,2})$ be a generalized Calabi-Yau metric manifold with pure spinors $\rho_{1,2}$. Let $\zeta^\pm$ be the corresponding global sections of $\det {L_1^\pm}^*$ written as in \eqref{eq:4.3}. Recall that the global sections $\zeta_{1,2}$ of $\det L_{1,2}^*$ defined by the pure spinors $\rho_{1,2}$ are given in these frames by $e^{\eta^+ + \eta^-}$ and $e^{\eta^+ - \eta^-}$. The global sections $J_1$ and $J_2$ of $U^\mathrm{ch}(E)$ constructed by Lemma \ref{lem:4.1} give rise to global sections
\begin{equation}
J^\pm = \frac{i}{2} \sum_i e^i_\pm e^\pm_i + i T \eta^\pm~.
\label{eq:4.4}
\end{equation}
\begin{rem}
We remark here than unlike the usual Calabi-Yau case with $H=0$, where the quantum corrections to the local fields $e^\alpha_\pm e_\alpha^\pm$ are given by a global holomorphic volume form, in this more general case, the relevant
  global sections of $\det T^{*1,0}_{\pm}$ are not given by the holomorphic expressions \eqref{eq:4holom1} but rather the covariantly constant ones \eqref{eq:4holom4}, where the dilaton correction appears explicitly. 
\label{rem:423}
\end{rem}
Let us now
recal the main theorem of \cite{heluani9} 
 in a slightly more general form. 
\begin{thm}[{\cite[Thm 5.5]{heluani9}}] Let $(M, \mathcal{J})$ be a generalized Calabi-Yau manifold. Let $\{e_i\}$ be a local frame for the associated Lie algebroid $L$ and let $\{e^i\}$ be the dual frame. Let $\zeta \in C^\infty(\det L^*)$ be a global section written as \eqref{eq:4.1} and let $J$ be the corresponding global section of $U^\mathrm{ch}(E)$ given by Lemma \ref{lem:4.1}. 
The following is true\footnote{Here $H$ is a section of $U^\mathrm{ch}(E)$ and should not be confused with the three form defining $E$. We keep this notation so that it agrees with previous literature.}. 
\begin{enumerate}
\item 
\begin{equation}
{[J}_\Lambda J] = - \left( H + \frac{c}{3} \lambda \chi \right), \qquad c = 3 \dim M
\label{eq:7.thma3}
\end{equation}
where 
\begin{multline}
H = H_0 - i T \mathcal{J} \mathcal{D} \eta = \frac{1}{4} \Bigl[ e^i \bigl( e^j[e_i,e_j] + e_i \bigl(e_j[ e^i, e^j] \bigr) \Bigr] - \frac{i}{2} T \mathcal{J} [e^i,e_i] -\\ \frac{1}{2} \left( e_i Se^i + e^i Se_i \right) - i T \mathcal{J} \mathcal{D} \eta
\label{eq:7.thma4}
\end{multline}
\item The fields $J$ and $H$ generate the $N=2$ superconformal vertex algebra of central charge $c$.
\end{enumerate}
\label{thm:7.thma}
\end{thm}
\begin{proof}
Define $J_0 = J - iT \eta$. In \cite{heluani9} it was proved that in the coordinate system where the global section of $\det L^*$ is constant  then the fields $J_0$ and $H_0$ generate a copy of the $N=2$ superconformal algebra. Since we will have to deal below with generalized Calabi-Yau metric manifolds, where there are two generalized complex structures and therefore two global spinors, we need to keep track of the fields $J$ in a more general coordinate system.  We compute
\begin{equation}
{[i T \eta}_\Lambda J_0] = \frac{1}{2} \lambda \left( \eta^{,i} e_i - \eta_{,i} e^i \right) = - \frac{i}{2}  \lambda \mathcal{J} \left( \eta^{,i}e_i + \eta_{,i} e^i \right) =  - i  \lambda \mathcal{J} \mathcal{D} \eta.
\label{eq:7.thma5}
\end{equation}
By skewsymmetry we obtain 
\begin{equation}
{[J}_\Lambda i T \eta] = i (\lambda + T) \mathcal{J} \mathcal{D} \eta. 
\label{eq:7.thma6}
\end{equation}
Combining \eqref{eq:7.thma6} and \eqref{eq:7.thma5} we obtain \eqref{eq:7.thma4}. The theorem follows from \cite[Thm 5.5]{heluani9}  since $J$ and $H$ are just expressions for $J_0$, $H_0$ in a more general coordinate system.
\end{proof}

Let now  $(M, \mathcal{J}_1, \mathcal{J}_2)$ be a generalized Calabi-Yau metric manifold. Choose frames $\{e_\alpha^\pm\}$ for the Lie algebroids $L_1^\pm$ with dual frames $\{e^\alpha_\pm\}$. Let $(M,g, J_\pm)$ be the associated  bihermitian manifold, and let \eqref{eq:4.3} be the local expressions for the global sections of $\det{L_1^\pm}^*$. Recall that in terms of the bihermitian data we have 
\begin{equation}
\eta^\pm = \varphi^\pm + 2 \Phi
\label{eq:7.1}
\end{equation}
where $\Omega^\pm = e^{\varphi^\pm} dz_\pm^1 \wedge \dots \wedge dz_\pm^n$ are the global holomorphic volume forms \eqref{eq:4holom1} and $\Phi$ is the dilaton.
We have the global sections \eqref{eq:4.4} of $U^\mathrm{ch}(E)$. 
\begin{thm}
The sections $J^\pm$ generate two commuting copies of the $N=2$ superconformal vertex algebra of central charge $c = \tfrac{3}{2} \dim M$. More precisely, defining 
\begin{multline}
H^\pm = \frac{1}{4} \left( e^\alpha_\pm \bigl( e_\beta^\pm [e_\alpha^\pm, e^\beta_\pm] \bigr) - e^\alpha_\pm \bigl( e^\beta_\pm [e_\alpha^\pm, e_\beta^\pm] \bigr) + e_\alpha^\pm \bigl( e^\beta_\pm [e^\alpha_\pm, e_\beta^\pm] \bigr) - e_\alpha^\pm \bigl( e_\beta^\pm [e^\alpha_\pm, e^\beta_\pm] \bigr) \right) \\ + \frac{1}{2} \left( e^\alpha_\pm Se_\alpha^\pm + e_\alpha^\pm Se^\alpha_\pm \right) - 
i \frac{T}{2} \mathcal{J}_\pm [e^\alpha_\pm, e_\alpha^\pm] - i T \mathcal{J}_\pm \mathcal{D} \eta^\pm~,
\label{eq:7.thm1}
\end{multline}
where $\mathcal{J}_\pm = \tfrac{1}{2} (\mathcal{J}_1 \pm \mathcal{J}_2 )$, 
we obtain the commutation relations:
\begin{equation}
\begin{aligned}
{[J^\pm}_\Lambda J^\pm] &= - \left( H^\pm + \frac{c}{3} \lambda \chi \right), & {[J^\pm}_\Lambda J^{\mp}] &= 0, \\
{[H^\pm}_\Lambda J^\pm] &= (2 T + 2\lambda + \chi S) J^\pm, & {[H^\pm}_\Lambda J^\mp] &= 0, \\
{[H^\pm}_\Lambda H^\pm] &= (2 T + 3 \lambda + \chi S) H^\pm + \frac{c}{3} \lambda^2 \chi, & {[H^\pm}_\Lambda H^\mp] &=0,
\end{aligned}
\label{eq:7.thm2}
\end{equation}
\label{thm:1}
\end{thm}
\begin{proof}
The proof of the Theorem will consist on 4 parts. First we will show that the two sectors commute, that is
\begin{equation}
{[J^+}_\Lambda J^-] = 0.
\label{eq:7.1a}
\end{equation}
In the second part we will identify the superconformal vector in each sector, namely we will compute $H^\pm$ given by the first equation in \eqref{eq:7.thm2}.
In the third part we will use the main Theorem of \cite{heluani9}, namely that each  $J_1 = J^+ + J^-$ and $J_2 = J^+ - J^-$ generate a copy of the $N=2$ superconformal vertex algebra of central charge $3 \dim M$. The superconformal vector of these algebras coincide and we will identify it with $H^+ + H^-$. After this, the theorem will follow by an application of the Jacobi identity of conformal algebras.  

\begin{enumerate}
\item \textbf{Commuting sectors}
\begin{multline}
{[e^\alpha_-}_\Lambda J^+] = - \frac{i}{2} \left( d^{\beta \alpha}_\gamma  e^\gamma_+ \right) e^+_\beta - \frac{i}{2} \left( e^{\beta \alpha}_\gamma e^\gamma_- \right) e^+_\beta  + \frac{i}{2} e^\beta_+ \left( d^{\gamma \alpha}_\beta e_\gamma^+ \right) +  \\ \frac{i}{2} e^\beta_+ \left( e^\alpha_{\beta \gamma} e^\gamma_- \right)  + i (\lambda + T) \eta^{+,\alpha^-} - i \lambda d^{\beta \alpha}_\beta = \\ = \frac{i}{2} \Bigl( e_{\beta\gamma}^\alpha \left( e^\beta_+ e^\gamma_- \right) - e^{\gamma \alpha}_\beta \left( e^\beta_- e_\gamma^+ \right) \Bigr) + i (\lambda + T) \left( \eta^{+,\alpha^-} - d^{\beta\alpha}_\beta \right) 
\label{eq:blablah}
\end{multline}
where we used quasi-associativity in the last line. According to \eqref{eq:6.14} the last term vanishes, and using skew symmetry we obtain:
\begin{equation}
{[J^+}_\Lambda e^\alpha_-] = \frac{i}{2} \Bigl( e_{\beta\gamma}^\alpha \left( e^\beta_+ e^\gamma_- \right) - e^{\gamma \alpha}_\beta \left( e^\beta_- e_\gamma^+ \right) \Bigr)~.
\label{eq:7.3}
\end{equation}
Similarly we compute
\begin{multline}
{[e_\alpha^-}_\Lambda J^+] = \frac{i}{2} \left( e^{\beta \gamma}_\alpha e_\gamma^- \right) e_\beta^+ + \frac{i}{2} \left( d_{\gamma \alpha}^\beta e^\gamma_+ \right) e_\beta^+ - \frac{i}{2} e^\beta_+ \left( d_{\beta \alpha}^\gamma e_\gamma^+ \right)  - \frac{i}{2} e^\beta_+ \left( e_{\beta \alpha}^\gamma e_\gamma^- \right) + \\ i \lambda d^\beta_{\beta \alpha} + i (\lambda + T) \eta^+_{,\alpha^-} = \\ \frac{i}{2} \left( e^{\beta \gamma}_\alpha \left( e_\gamma^- e_\beta^+ \right) - e_{\gamma \alpha}^\beta \left(e^\gamma_+ e_\beta^- \right) \right) + i (\lambda + T) \left( d_{\beta\alpha}^\beta + \eta^+_{,\alpha^-} \right)~,
\label{eq:7.4}
\end{multline}	
and the last term vanishes because of \eqref{eq:6.13}. Using skew-symmetry we obtain:
\begin{equation}
{[J^+}_\Lambda e^-_\alpha] =  \frac{i}{2} \left( e^{\beta \gamma}_\alpha \left( e_\gamma^- e_\beta^+ \right) - e_{\gamma \alpha}^\beta \left(e^\gamma_+ e_\beta^- \right) \right)~.
\label{eq:7.5}
\end{equation}
We also need
\begin{equation}
{[T \eta^-}_\Lambda J^+] = - \lambda \frac{i}{2} \left( \eta^{-,\alpha^+} e_\alpha^+ - \eta^-_{,\alpha^+} e^\alpha_+ \right)~,
\label{eq:7.6}
\end{equation}
and from skewsymmetry:
\begin{equation}
{[J^\pm}_\Lambda i T \eta^-] = - \frac{1}{2} (\lambda + T)   \left( \eta^{-,\alpha^+} e_\alpha^+ - \eta^-_{,\alpha^+} e^\alpha_+ \right).
\label{eq:7.7}
\end{equation}
Combining \eqref{eq:7.7}, \eqref{eq:7.5} and \eqref{eq:7.3} we obtain
\begin{multline}
{[J^+}_\Lambda J^-] = - \frac{1}{4} \Bigl( e_{\beta\gamma}^\alpha \left( e^\beta_+ e^\gamma_- \right) - e^{\gamma \alpha}_\beta \left( e^\beta_- e_\gamma^+ \right) \Bigr) e_\alpha^- +  \\ + \frac{1}{4} e^\alpha_- \left( e^{\beta \gamma}_\alpha \left( e_\gamma^- e_\beta^+ \right) - e_{\gamma \alpha}^\beta \left(e^\gamma_+ e_\beta^- \right) \right) - \frac{1}{2} (\lambda + T)  \left( \eta^{-,\alpha^+} e_\alpha^+ - \eta^-_{,\alpha^+} e^\alpha_+ \right) \\ - \frac{1}{4} \int_0^\Lambda \Bigl[  e_{\beta\gamma}^\alpha \left( e^\beta_+ e^\gamma_- \right) - e^{\gamma \alpha}_\beta \left( e^\beta_- e_\gamma^+ \right)_\Gamma e^-_\alpha \Bigr] d\Gamma
\label{eq:7.8}
\end{multline}
The integral term can be easily evaluated to be 
\begin{equation}
- \frac{\lambda}{2} \Bigl( e_{\beta \alpha}^\alpha e^\beta_+ + e^{\beta \alpha}_\alpha e_\beta^+ \Bigr),
\label{eq:7.9}
\end{equation}
which cancels the $\lambda$-term in \eqref{eq:7.8} due to \eqref{eq:6.13} and \eqref{eq:6.14}.  Using quasi-associativity we can write the first two terms of \eqref{eq:7.8} (the cubic terms) as:
\begin{equation}
- \frac{1}{2} T \left( e^{\beta \alpha}_\alpha e^+_\beta + e^{\alpha}_{\beta\alpha} e^\beta_+ \right) =  \frac{1}{2} T \left( \eta^{-,\beta^+} e_\beta^+ - \eta^-_{,\beta^+} e^\beta_+  \right)
\label{eq:7.10}
\end{equation}
Combining \eqref{eq:7.10}, \eqref{eq:7.9} and \eqref{eq:7.8} we obtain \eqref{eq:7.1a}.
\item \textbf{Superconformal vectors.} 
We may work with the frames \eqref{eq:b4.7} and \eqref{eq:b4.8} where $c^{\pm \alpha}_{\beta \gamma} = 0$. 
\begin{multline}
{[e_\alpha^+}_\Lambda J^+] =\frac{i}{2} \left(d^{\beta \gamma}_\alpha e_\gamma^-  - d_{\alpha \gamma}^\beta e^\gamma_-  \right) e_\beta^+  + \\ i \chi e_\alpha^+ + i (\lambda + T) \eta^+_{,\alpha^+} + \frac{i}{2} \int_0^\Lambda \bigl[ \left( d^{\beta \gamma}_\alpha e_\gamma^- - d_{\alpha \gamma}^\beta {e^\gamma_-}\right)_\Gamma e_\beta^+ \bigr] d\Gamma
\label{eq:7.2.1}
\end{multline}
The integral term clearly vanishes. Using quasi-associativity we obtain
\begin{equation}
{[e_\alpha^+}_\Lambda J^+] = \frac{i}{2} \left( d^{\beta \gamma}_\alpha (e_\gamma^- e_\beta^+) - d_{\alpha \gamma}^\beta (e^\gamma_- e_\beta^+) \right) + i \chi e_\alpha^+ + i (\lambda + T) \eta_{,\alpha^+}^+~,
\label{eq:7.2.2}
\end{equation}
and with skew-symmetry this reads:
\begin{equation}
{[J^+}_\Lambda e_\alpha^+] = \frac{i}{2} \left( d^{\beta \gamma}_\alpha (e_\gamma^- e_\beta^+) - d_{\alpha \gamma}^\beta (e^\gamma_- e_\beta^+) \right) - i (\chi + S) e_\alpha^+ - i \lambda \eta_{,\alpha^+}^+~.
\label{eq:7.2.3}
\end{equation}
Similarly we compute 
\begin{multline}
{[e^\alpha_+}_\Lambda J^+] = \frac{i}{2} e^\beta_+ \left( d_{\beta \gamma}^\alpha e^\gamma_- - d^{\alpha \gamma}_\beta e_\gamma^- \right) - i \chi e^\alpha_+ + i (\lambda + T) \eta^{+, \alpha^+} = \\ \frac{i}{2} \left( d_{\beta\gamma}^\alpha (e^\beta_+ e^\gamma_-) - d^{\alpha \gamma}_\beta (e^\beta_+ e_\gamma^-) \right) - i \chi e^\alpha_+ + i (\lambda + T) \eta^{+,\alpha^+}
\label{eq:7.2.4}
\end{multline}
and using skew-symmetry:
\begin{equation}
{[J^+}_\Lambda e^\alpha_+]= \frac{i}{2} \left( d_{\beta\gamma}^\alpha (e^\beta_+ e^\gamma_-) - d^{\alpha \gamma}_\beta (e^\beta_+ e_\gamma^-) \right) + i (\chi + S) e^\alpha_+ - i \lambda \eta^{+, \alpha^+}~.
\label{eq:7.2.5}
\end{equation}
We also need
\begin{equation}
{[i T \eta^+}_\Lambda J^+] = \frac{\lambda}{2} \left( \eta^{+,\alpha^+} e_\alpha^+ - \eta^+_{,\alpha^+} e^\alpha_+ \right)~,
\label{eq:7.2.6}
\end{equation}
from where we get
\begin{equation}
{[J^+}_\Lambda i T\eta^+] = - \frac{1}{2} (\lambda + T) \left( \eta^{+,\alpha^+} e_\alpha^+ - \eta^+_{,\alpha^+} e^\alpha_+ \right)~.
\label{eq:7.2.7}
\end{equation}
We can now compute using the non-commutative Wick formula:
\begin{multline}
{[J^+}_\Lambda J^+] = - \frac{1}{4} \left( d^\alpha_{\beta \gamma} (e^\beta_+ e^\gamma_-) - d^{\alpha \gamma}_\beta (e^\beta_+ e_\gamma^-) \right) e_\alpha^+  + \frac{1}{4} e^\alpha_+ \left( d^{\beta \gamma}_\alpha (e_\gamma^- e_\beta) - d_{\alpha\gamma}^\beta (e^\gamma_- e_\beta^+) \right) \\ - \frac{1}{2} \Bigl((\chi + S) e^\alpha_+ \Bigr) e_\alpha^+ - \frac{1}{2} e^\alpha_+ (\chi + S) e_\alpha^+ + \frac{1}{2} \lambda \left( \eta^{+,\alpha^+} e_\alpha^+ - \eta^+_{,\alpha^+} e^\alpha_+ \right)\\ - \frac{1}{2} (\lambda + T) \left(\eta^{+,\alpha^+} e^+_\alpha - \eta_{,\alpha^+}^+ e^\alpha_+\right) \\   - \frac{1}{4} \int_0^\Lambda \Bigr[ d_{\beta \gamma}^\alpha (e^\beta_+ e^\gamma_-) - d^{\alpha \gamma}_\beta {(e^\beta_+ e_\gamma^-)}_\Gamma e_\alpha^+ \Bigr] d \Gamma - \frac{1}{2} \int_0^\Lambda \bigl[ {(\chi + S) e^\alpha_+}_\Gamma e_\alpha^+ \bigr] d\Gamma
\label{eq:7.2.8}
\end{multline}
Since the fields $e_\alpha^+$ and $e^\alpha_+$ are odd, we see that the $\chi$-terms vanish. The integral terms are easily evaluated:
\begin{equation}
 \frac{1}{2} \lambda \left( d_{\alpha\beta}^\alpha e^\beta_- - d^{\alpha \beta}_\alpha e_\beta^- \right) - \frac{1}{2} \lambda [e^\alpha_+, e_\alpha^+] - \frac{1}{2} \chi \lambda \dim M = - \frac{1}{2} \lambda \chi \dim M.
\label{eq:7.2.9}
\end{equation}
Therefore the $\lambda$-terms in \eqref{eq:7.2.8} vanish. Using quasi-associativity we can write the first two terms of \eqref{eq:7.2.8} as 
\begin{equation}
- \frac{1}{2} e^\beta_+ \bigl( e_\alpha^+ [e_\beta^+, e^\alpha_+] \bigr) - \frac{T}{2} \left( d^{\alpha \beta}_\alpha e_\beta^- - d^\alpha_{\alpha\beta} e^\beta_- \right)~.
\label{eq:7.2.10}
\end{equation}
Collecting terms and using quasi-commutativity we obtain
\begin{multline}
{[J^+}_\Lambda J^+] = - \frac{1}{4} \left( e^\alpha_+ \bigl( e_\beta^+ [e_\alpha^+, e^\beta_+] \bigr) + e_\alpha^+ \bigl( e^\beta_+ [e^\alpha_+, e_\beta^+] \bigr) \right) \\ - \frac{1}{2} \left( e^\alpha_+ Se_\alpha^+ + e_\alpha^+ Se^\alpha_+ \right) - \frac{1}{2} T \left( \eta^{+,\alpha^+} e_\alpha^+ - \eta^+_{,\alpha^+} e^\alpha_+ \right) - \frac{1}{2} \lambda\chi \dim M~.
\label{eq:7.2.11}
\end{multline}
Which in this frame coincides with the first equation of \eqref{eq:7.thm2}. A similar computation holds in the minus sector. The computation in the more general frame where $c_{\alpha \beta}^{\pm \gamma} \neq 0$ is similar. 
\item \textbf{The diagonal embedding.}\\
Since the manifold $(M, \mathcal{J}_1)$ is generalized Calabi-Yau, Theorem \ref{thm:7.thma} says that $J_1 = J^+ + J^-$ generates an $N=2$ superconformal vertex algebra of central charge $c = 3 \dim M$. The superconformal vector is given by
\eqref{eq:7.thma4}, 
where $\{e_i\}$ is a frame for $L_1$ and $\{e^i\}$ is the dual frame. We can consider the frame given by $\{e^+_\alpha\} \cup \{e^-_\alpha\}$ and their corresponding dual frames. Similarly the section $J_2 = J^+ - J^-$ generates another $N=2$ superconformal vertex algebra of central charge $c = 3 \dim M$. The superconformal vector is given by an expression like \eqref{eq:7.thma4} where we now use the frame for $L_2$ given by $\{e^+_\alpha\} \cup \{e^\alpha_-\}$. We want to show that this field $H$ is actually $H^+ + H^-$. 

Recall that we have chosen the section 
\[ e^{4 \Phi + \varphi^+ + \varphi^-} e^1_+ \wedge \dots \wedge e^n_+\wedge e^1_- \wedge \dots \wedge e^n_- = e^{\eta^+ + \eta^-} e^1_+ \wedge \dots \wedge e^n_+\wedge e^1_- \wedge \dots  \wedge e^n_-\] of $\det L_1^*$. In the frame $\{e^+_\alpha\} \cup \{e^-_\alpha\}$ for $L_1$, and using the expressions in \eqref{eq:b4.7} and \eqref{eq:b4.8} such that $c_{\alpha\beta}^{\pm \gamma} = 0$ we have:
\begin{multline}
H_1 = \frac{1}{4} \Bigl( e^\alpha_+ \bigl(e^\beta_- [e_\alpha^+, e_\beta^-] \bigr) + e^\alpha_- \bigl(e^\beta_+ [e_\alpha^-, e_\beta^+] \bigr) + e_\alpha^+ \bigl( e_\beta^- [e^\alpha_+, e^\beta_-] \bigr) \\ +  e_\alpha^- \bigl(e_\beta^+ [e_\alpha^-, e^\beta_-] \bigr) \Bigr)  - \frac{T}{2} \Bigl( d_{\alpha\beta}^\alpha e^\beta_- + d_{\alpha}^{\alpha \beta} e_\beta^- - e_{\beta \alpha}^\alpha e^\beta_+ - e^{\beta\alpha}_\alpha e_\beta^+ \Bigr)  \\ + \frac{1}{2} \left( e^\alpha_+ Se_\alpha^+ + e^\alpha_- Se_{\alpha^-} + e_\alpha^+ Se^\alpha_+ + e_\alpha^- Se^\alpha_- \right) - i T \mathcal{J}_1 \mathcal{D}  ( 4 \Phi + \varphi^+ + \varphi^-)  \label{eq:7.3.1}
\end{multline}
The first term can be written as 
\begin{equation}
 \frac{1}{2} \left( e^\alpha_+ \bigl(e_\beta^+ [e_\alpha^+,e^\beta_+] \bigr) + e^\alpha_- \bigl(e_\beta^- [e_\alpha^-, e^\beta_-] \bigr) \right),
\label{eq:7.3.2}
\end{equation}
while the second term is 
\begin{multline}
- \frac{T}{2} \left( \eta^{+, \beta^-} e_\beta^- - \eta^{+}_{,\beta^-} e^\beta_- + \eta^{-,\beta^+} e_\beta^+ - \eta^{-}_{,\beta^+} e^\beta_+ \right) = i T \mathcal{J}_1 \mathcal{D} \Bigl( \eta^+ + \eta^- \Bigr) + \\ \frac{T}{2} \Bigl( \eta^{+,\beta^+} e_\beta^+ - \eta^{+}_{,\beta^+} e^\beta_+ + \eta^{-,\beta^-} e_\beta^- - \eta^-_{,\beta^-} e^\beta_- \Bigr)
\label{eq:7.3.3}
\end{multline}
From where we easily see $H_1 = H^+ + H^-$. We can perform a similar computation using the global section 
\begin{equation*}
e^{\varphi^+ - \varphi^-} e^1_+ \wedge \dots \wedge e^n_+ \wedge e_1^- \wedge \dots \wedge  e_n^-  =
e^{\eta^+ - \eta^-} e^1_+ \wedge \dots \wedge e^n_+ \wedge e_1^- \wedge \dots \wedge  e_n^-
\end{equation*}
of $\det L_2^*$ and the frame $\{e_\alpha^+\} \cup \{e^\alpha_-\}$ of $L_2$ to obtain that $H_2 = H_1 = H^+ + H^-$. 
\item  \textbf{The $N=2,2$ algebra.}\\
The Theorem follows from the following Lemma that in turn is a application of the Jacobi identity for SUSY Lie conformal algebras \cite[Proof of Thm 6.2]{heluani8}:
\begin{lem}
Let $J^+$ and $J^-$ be two commuting superfields satisfying 
\begin{equation}
{[J^\pm}_\Lambda J^\pm] = - \left( H^\pm + \frac{c}{3} \lambda \chi \right), 
\label{eq:7.4.1}
\end{equation}
for some $c \in \mathbb{C}$ and some odd fields $H^\pm$. Let $H = H^+ + H^-$ and suppose moreover that both pairs  $(J^+ + J^-, H)$ and $(J^+ - J^-, H)$ generate the $N=2$ superconformal vertex algebra of central charge $2c$. Then the quadruple $J^\pm, H^\pm$ generates two commuting copies of the $N=2$ superconformal vertex algebra of central charge $c$. 
\label{lem:7.thmb1}
\end{lem}
\end{enumerate}
\end{proof}
\begin{ex}
In the usual Calabi-Yau case, these two commuting structures agree with the ones constructed in \cite{heluani8}. Indeed in this case we have that both hermitian complex structures agree $J_+ = J_- = J$ and we can therefore choose a holomorphic coordinate system for both $\{z^\alpha_\pm\} = \{z^\alpha\}$. Note also that we may work in the coordinates where the holomorphic volume form is constant.  In this coordinate system, the sections $H^\pm$ and $J^\pm$ can be written as:
\begin{equation}
\begin{aligned}
J_1 &:= J^+ + J^- = i SB^\alpha \Psi_\alpha - i SB^{\bar \alpha} \Psi_{\bar \alpha}, \\
J_2 &:= J^+ - J_- = \omega^{\alpha \bar \beta} \Psi_\alpha \Psi_{\bar \beta} + \omega_{\alpha \bar \beta} SB^\alpha SB^{\bar \beta}, \\
H &:=  H^+ + H^- = SB^i S\Psi_i + TB^i \Psi_i, \\
H'&:=  H^+ - H^- =  \left( \Gamma^{\alpha}_{\varepsilon \gamma} g^{\varepsilon \bar \beta} SB^\gamma \right) \left( \Psi_{\bar \beta} \Psi_\alpha \right) + g^{\alpha \bar \beta} S\Psi_\alpha \Psi_{\bar \beta}  \\ & \quad +\left( \Gamma_{\bar\varepsilon \bar \gamma}^{\bar \beta} g^{\alpha \bar \varepsilon} SB^{\bar \gamma} \right) \left( \Psi_\alpha \Psi_{\bar \beta} \right) + g^{\alpha \bar \beta} \Psi_\alpha S\Psi_{\bar \beta} \\ & \quad + g_{\alpha \bar \beta} TB^\alpha SB^{\bar \beta} + g_{\alpha \bar \beta} SB^\alpha TB^{\bar\beta}. 
\end{aligned}
\label{eq:alamierda}
\end{equation}
\label{ex:mio}
\end{ex}
\section{Half-twisted model} \label{sec:7}
In this section we study the \emph{topological twist} of the $N=2$ algebras described in the previous section. We first recall the double complex computing the generalized Hodge decomposition of a generalized K\"ahler manifold \cite{gualtieri3}.

Let $(M, \mathcal{J}_1, \mathcal{J}_2)$ be a generalized K\"ahler manifold. $\mathcal{J}_1$ induces a decomposition of forms into its eigenspaces \eqref{eq:ukn}. Since $\mathcal{J}_2$ also acts on $\wedge^\bullet T^*$ via the spin representation and commutes with the action of $\mathcal{J}_1$, $U_k$ in turn is decomposed as
\begin{equation}
U_k = U_{k,|k|-n}\oplus U_{k,|k|-n+2}\oplus \cdots \oplus
U_{k,n-|k|},
\label{eq:ukpq}
\end{equation}
where $U_{p,q}$ is the intersection of the $ip$-eigenspace of
$\mathcal{J}_1$ and the $iq$-eigenspace of $\mathcal{J}_2$. 

Recall that, with respect to $\mathcal{J}_1$, the (H-twisted) de Rham differential decomposes as $d_H = \partial_1 + \overline{\partial_1}$, and these differentials act on
\begin{equation}
\xymatrix{C^\infty(U_k)\ar@<0.5ex>[r]^{\overline{\partial}_1}&C^\infty(U_{k+1})\ar@<0.5ex>[l]^{\partial_1}}.
\label{eq:8.1}
\end{equation}
This decomposition is further refined by the action of $\mathcal{J}_2$ into 
\begin{equation}
d_H = \delta_+ + \delta_- + \overline{\delta}_+ + \overline{\delta}_-,
\label{eq:8.2}
\end{equation}
with these operators defined by:
\begin{equation}
\xymatrix{
U_{p-1,q+1} & &U_{p+1,q+1} \\
 &U_{p,q}\ar@{.>}[r]^{\overline{\partial}_1}\ar@{.>}[l]^{\partial_1}\ar@{.>}[u]^{\overline{\partial}_2}\ar@{.>}[d]^{\partial_2}\ar[lu]^{\delta_-}\ar[ru]^{\overline{\delta}_+}\ar[ld]^{\delta_+}\ar[rd]^{\overline{\delta}_-} & \\
 U_{p-1,q-1}& & U_{p+1,q-1}
}
\label{eq:8.3}
\end{equation}
where $\overline{\partial_1} = \overline{\delta_+} + \overline{\delta_-}$ and $\overline{\partial_2} = \overline{\delta_+} + \delta_-$. 
This decomposition implies:
\begin{prop}[generalized K\"ahler identities.]
For a generalized K\"ahler structure, we have the identities
\[
\overline{\delta}_+^* = -\delta_+\ \ \text{and}\ \ \overline{\delta}_-^* = \delta_-.
\]
\end{prop}
Obtaining thus the following relation between all available Laplacians:
\begin{equation}
\Delta_{d_H} = 2 \Delta_{\overline{\partial}_{1/2}} = 2 \Delta_{\partial_{1/2}} = 4 \Delta_{\overline{\delta}_{\pm}} = 4 \Delta_{\delta_\pm}.
\label{eq:8.5}
\end{equation}
Note that by acting on $\rho_1$ we obtain an isomorphism of sheaves:
\begin{equation}
U_{p, q} \simeq \wedge^{r} \overline{L_1^+} \otimes \wedge^s \overline{L_1^-}, \quad p + q + n = 2 r, \quad p-q +n = 2s 
\label{eq:menostodavia}
\end{equation}
and similarly applying to $\rho_2$ we obtain:
\begin{equation}
U_{p,q} \simeq \wedge^r \overline{L_1^+} \otimes \wedge^s L_1^-, \qquad p+q+n= 2r, \quad q -p + n = 2 s~.
\label{eq:menostodavia2}
\end{equation}
In the generalized Calabi-Yau metric case, these isomorphisms are in fact isomorphisms of bi-complexes. Indeed since $L_1^\pm$ are closed under the Lie bracket, we have $d_{L^1} = d_{L_1^+} + d_{L_1^-}$. Since we know that under \eqref{eq:menostodavia} the differential $\overline{\partial_1}$ is mapped to $d_{L_1}$, by degree considerations we obtain $\overline{\delta_\pm}$ is mapped to $d_{L_1^\pm}$. We arrive to 
\begin{prop}
On a generalized Calabi-Yau metric manifold, the isomorphisms \eqref{eq:menostodavia} and \eqref{eq:menostodavia2} are isomorphisms of bicomplexes. The corresponding spectral sequences are degenerate at $E_2$ and, in the case of compact $M$, converge to the $H$-twisted de Rham cohomology of $M$. 
\label{prop:htwisteddeRham}
\end{prop}
The last statement in the above proposition is true for the complex $U_{\bullet,\bullet}$ on any compact generalized K\"ahler manifold \cite{gualtieri3}.

We are now in position to study an \emph{affine generalization} of all of these complexes. For the moment we let $(M, \mathcal{J}_1, \mathcal{J}_2)$ be a generalized Calabi-Yau metric manifold. In the vertex algebra of global sections of $U^\mathrm{ch}(E)$ there are several bigradings corresponding to different choices of conformal vectors and $U(1)$ currents. In the previous sections, in order to accomodate supersymmetry, it was natural to consider the Virasoro field arising from the decomposition:
\begin{equation}
H(z,\theta) = H_1 = H^+ + H^- = \bigl( G^+(z) + G^{-}(z) \bigr) + 2 \theta L(z)~. 
\label{eq:8.6}
\end{equation}
With respect to this field, the basic fermions $e_\alpha^\pm$ and $e^\alpha_\pm$ are of conformal weight $1/2$. We may consider different $U(1)$ currents giving rise to different charge decompositions. 

We will now perform what is called a topological twist. This consists of changing the conformal weights and $U(1)$-charge of the fields by considering different Virasoro and $U(1)$-currents. Since we have two commuting copies of the $N=2$ superconformal algebra, we may perform this twisting in each sector (plus or minus) independently. For a given $N=2$ structure $(J,H)$  we may consider the operator 
\begin{equation}
L_0 = \frac{1}{2} \left( H_{(1|0)} \pm i J_{(0|1)} \right)~,
\label{eq:8.7}
\end{equation}
which acts diagonally and its eigenvalues will be called the conformal weights of the corresponding states. The eigenvalues of $J_0 := - i J_{(0|1)}$ will be called \emph{charge}. We obtain two possible twistings by choosing different signs in \eqref{eq:8.7}. 

Since we have two commuting copies of the $N=2$ superconformal algebra in $U^\mathrm{ch}(E)$, we may use either the same or different signs in the plus or minus sector. It is customary to call the former an $B$-twist, and the latter a $A$-twist.  Consider the operators 
\begin{equation}
\begin{aligned}
L^\pm_0 &:= \frac{1}{2} \left( H^\pm_{(1|0)} + i J^\pm_{(0|1)} \right) & J^\pm_0 &:= -i J^\pm_{(0|1)} \\ Q_0^\pm &:= \frac{1}{2} \left( H^\pm_{(0|1)} + i J^\pm_{(0|0)} \right) & G_0^\pm &:= \frac{1}{2} \left( H^\pm_{(0|1)} - i J^\pm_{(0|0)} \right)
\end{aligned}
\end{equation}

\begin{prop}
Consider the embedding $i: U_{\bullet,\bullet} \hookrightarrow U^\mathrm{ch}(E)$ obtained by composing \eqref{eq:menostodavia} with the obvious embedding  $\wedge^\bullet \overline{L_1} \hookrightarrow U^\mathrm{ch}(E)$. We have 
\begin{equation}
Q^\pm_0 \circ i = i \circ \overline{\delta_\pm}.
\label{eq:diff1}
\end{equation}
Similarly, consider the embedding $j: U_{\bullet,\bullet} \hookrightarrow \wedge^\bullet \overline{L_2} \subset U^\mathrm{ch}(E)$, we obtain
\begin{equation}
Q^+_0 \circ j = j \circ \overline{\delta_+}, \qquad G^-_0 \circ j = j \circ \delta_-
\label{eq:diff2}
\end{equation}
\label{prop:diff1}
\end{prop}
\begin{proof}
Since $(M, \mathcal{J}_1)$ is generalized Calabi-Yau, the map \eqref{eq:isomalgebroid} is an isomorphism of complexes. The differential $\overline{\partial_1} = \overline{\delta_+} + \overline{\delta_-}$ is mapped therefore to $d_{L^1}$. A simple computation \cite{heluani9} shows that 
$Q^B_0 := Q^+_0 + Q^-_0$ equals $d_{L^1}$ when restricted to $\wedge^\bullet \overline{L_1} = \wedge^\bullet \overline{L_1^+} \otimes \wedge^\bullet \overline{L_1^-}$. The latter bigrading coincides with the one by eigenvalues of $J^\pm_0$. Since $Q^\pm_0$ increases the degree of the $\wedge^\bullet \overline{L_1^\pm}$ component, we obtain \eqref{eq:diff1} by degree considerations. \eqref{eq:diff2} follows in the same way by using $\mathcal{J}_2$ in place of $\mathcal{J}_1$. 
\end{proof}
\begin{rem}
Note that by considering the complex conjugate embeddings $\bar{i}$ and $\bar{j}$ we would obtain
\begin{equation}
G^\pm_0 \circ \bar{i} = \bar{i} \circ \delta_\pm, \qquad Q^-_0 \circ \bar{j} = \bar{j} \circ \overline{\delta_-} \qquad G^+_0 \circ \bar{j} = \bar{j} \circ \delta_+
\label{eq:diff3}
\end{equation}
\label{rem:diff2}
\end{rem}
\begin{rem}
The identification of the operators $Q^\pm$, $G^\pm$ with the differentials $\overline{\delta_\pm}$ and $\delta_\pm$ under certain embeddings of differential forms into the chiral-anti-chiral de Rham complex, shows that the full infinite dimensional Lie algebra given by the Fourier modes of the superfields $H^\pm, J^\pm$ should be viewed as an affine, or superconformal version of the generalized K\"ahler identities in the generalized Calabi-Yau case.
\label{rem:infinite}
\end{rem}

Note that in order to define the bigradings by eigenvalues of $L^\pm_0$ and $J^\pm_0$ and  the differentials $Q^\pm_0$ and $G^\pm_0$, we only need the zero modes of the fields $H^\pm, J^\pm$ and not the full superconformal algebra. It is easy to see that these zero modes are well defined on any generalized K\"ahler manifold since the quantum corrections involve derivatives of fields \eqref{eq:4.2}. We can therefore define:
\begin{defn}
Let $(M, \mathcal{J}_1, \mathcal{J}_2)$ be a generalized K\"ahler manifold. We define the \emph{chiral de Rham complex} of $M$ as the sheaf of super-vertex algebras given by the cohomology of the chiral-anti-chiral de Rham complex:
\begin{equation}
\Omega^\mathrm{ch}_M := \mathscr{H}^* \left( U^\mathrm{ch}(E), Q_0^+ \right).
\label{eq:8.8}
\end{equation}
\label{defn:8.1}
\end{defn}
The sheaf $\Omega^\mathrm{ch}_M$ is to $U^{\mathrm{ch}}(E)$ what the \emph{holomorphic chiral de Rham complex} of \cite{malikov} is to the \emph{smooth} chiral de Rham complex. In particular, we can give a description in terms of generators and relations just as in the usual case of \cite{malikov}. Recall that for a complex manifold $M$ we may define a holomorphic Courant algebroid $E$ just as in the smooth setting. The definition of $U^\mathrm{ch}(E)$ carries over without change to this case. In particular, when $E = T_{\mathbb{C}} \oplus T^*_{\mathbb{C}}$ is the standard Courant algebroid we obtain the usual holomorphic chiral de Rham complex. 

For a generalized complex manifold $M$, there is no such a notion as \emph{holomorphic tangent bundle} $T_\mathbb{C}$, so a priori there is no obvious way of constructing the analog of the holomorphic chiral de Rham complex. However, for a generalized K\"ahler manifold $M$ we have the following analog of $T_\mathbb{C}$. As mentioned above, since $L_1^\pm$ are closed under the Lie bracket, we may decompose $d_{L_1} = d_{L_1^+} + d_{L_1^-}$. We have the bicomplex
\begin{equation}
\wedge^{p,q} \overline{L_1} := \wedge^p \overline{L_1^+} \otimes \wedge^q \overline{L_1^-}
\label{eq:9.1}
\end{equation}
Define 
\begin{equation}
\Omega^q = \mathrm{ker} d_{L_1^+}: \wedge^{0,q} \overline{L_1} \rightarrow \wedge^{1,q} \overline{L_1}~.
\label{eq:9.2}
\end{equation}
A variation of the standard $\overline{\partial}$-lemma gives 
\begin{prop}
The following is a resolution of $\Omega^q$:
\begin{equation}
 \Omega^q \rightarrow \wedge^{0,q} \overline{L_1} \xrightarrow{d_{L_1^+}} \wedge^{1,q} \overline{L_1} \xrightarrow{d_{L_1^+}} \wedge^{2,q} \overline{L_1} \xrightarrow{d_{L_1^+}} \dots
\label{eq:9.3}
\end{equation}
\label{prop:}
\end{prop}
Similarly we may define $\Theta^q:= \mathrm{ker} \,d^*_{L_1^+}: \wedge^{0,q} L_1 \rightarrow \wedge^{1,q} L_1$. We have by restriction a non-degenerate pairing
\begin{equation}
\Theta^q \otimes \Omega^q \rightarrow \mathscr{O} := \Omega^0 \simeq \Theta^0 \subset C^\infty (M).
\label{eq:9.4}
\end{equation}
The sheaf $\mathscr{O}$ is a sheaf of rings on $M$ and it will play the role of the structure sheaf on a generalized K\"ahler manifold. The sheaves $\Theta^p$ and $\Omega^p$ are sheaves of $\mathscr{O}$-modules and  \eqref{eq:9.4} is $\mathscr{O}$-bilinear.  The sheaf $\mathcal{E} := \Theta^1 \oplus \Omega^1$ will play the role of the holomorphic Courant algebroid. In particular, we have by restriction of the usual operations a differential $\mathcal{D}:\mathscr{O} \rightarrow \mathcal{E}$, and we can define as always actions of of sections of $\mathcal{E}$ into sections of $\mathscr{O}$ by\footnote{This definition seems circular, but it agrees with the restriction of the usual smooth action.}
\begin{equation}
\pi(X) \cdot f := 2 \langle X, \mathcal{D} f \rangle
\label{eq:9.5}
\end{equation}
Finally note that we have a well defined Dorfman bracket $[\cdot,\cdot]: \mathcal{E} \otimes \mathcal{E} \rightarrow \mathcal{E}$, and that we can define the Courant bracket by \eqref{eq:dorf}.  
The following is straightforward
\begin{prop}
The data $(\mathcal{E}, \mathcal{D}, \langle\cdot, \cdot \rangle)$ satisfies (2)-(5) of Definition \ref{defn:1}. 
\end{prop}
We now arrive to the following description of $\Omega^\mathrm{ch}_M$. 
\begin{thm}
Let $(M, \mathcal{J}_1, \mathcal{J}_2)$ be a generalized K\"ahler manifold and let $(\mathcal{E}, \mathcal{D}, \mathscr{O})$ be as above. The sheaf $\Omega(E)$ is the sheaf of SUSY vertex algebras generated by functions $i : \mathscr{O} \hookrightarrow \Omega^\mathrm{ch}_M$ and sections of $\mathcal{E}$ declared to be odd, $j:\Pi\mathcal{E} \hookrightarrow \Omega^\mathrm{ch}_M$ subject to the relations  (1)-(5) of Proposition \ref{prop:universal}. 
\label{thm:holomorphiccdr}
\end{thm}
\begin{proof}
Let $\tilde{\Omega}_M^\mathrm{ch}$ be the sheaf described by the Theorem. The fact that $\tilde{\Omega}_M^\mathrm{ch}$ is  a well defined sheaf of SUSY vertex algebras is proved in the same way as in the usual smooth case. Note that we have an obvious inclusion $\tilde{\Omega}_M^\mathrm{ch} \hookrightarrow U^\mathrm{ch}(E)$ given by the inclusions $\mathscr{O} \subset C^\infty(M)$ and $\mathcal{E} \subset E$. Moreover, since $Q^+_0$ acts as $d_{L_1^+}$ in the restriction $\wedge^\bullet \overline{L_1} \hookrightarrow U^\mathrm{ch}(E)$ and also in the restriction $\wedge^\bullet \overline{L_2} \hookrightarrow U^\mathrm{ch}(E)$ we see that $\tilde{\Omega}^\mathrm{ch}_M \subset \ker	Q^+_0$. 
Note also that from
\begin{equation}
[G^+_0, Q^+_0] = L^+_0~,
\label{eq:9.6}
\end{equation}
we see that the cohomologies in \eqref{eq:8.8} are concentrated in conformal weight zero. Let $U^\mathrm{ch}(E)_0$ be the kernel of $L^+_0$. Note that the spectrum of $J^+_0$ is non-negative on $U^\mathrm{ch}(E)_0$. Indeed, the basic fields of negative charge $-1$ are sections of $L_1^+$, which in turn have positive conformal weight $1$ with respect to $L^+_0$. We clearly have 
\begin{equation}
\tilde{\Omega}^\mathrm{ch}_M \subset \ker J^+_0 \cap \ker Q^+_0. 
\label{eq:9.7}
\end{equation}
therefore we have a map $\tilde{\Omega}^\mathrm{ch}_M \hookrightarrow \Omega^\mathrm{ch}_M$. We can check that it is an isomorphism locally. Let 
\begin{equation}
\overline{\Omega}^p := \ker d_{L_1^-}: \wedge^{p,0} \overline{L_1} \rightarrow \wedge^{p,1} \overline{L_1}. 
\label{eq:9.8}
\end{equation}
As in the usual complex geometry case, we have that $(\overline{\Omega}^\bullet, d_{L_1^+})$ is a resolution of $\mathbb{C}$. We can locally write
\begin{equation}
U^\mathrm{ch}(E)_0 = \tilde{\Omega}^\mathrm{ch}_M \otimes \overline{\Omega}^\bullet
\label{eq:9.9}
\end{equation}
where the grading on the left is by $J^+_0$-charge and agrees with the grading on the right. The differential $Q^+_0$ acts as $1\otimes d_{L_1^+}$ and the result follows. 
\end{proof}
\begin{cor}
The cohomology in \eqref{eq:8.8} is concentrated in degree zero, namely, $U^\mathrm{ch}(E)_0$ is a resolution of $\Omega^\mathrm{ch}_M$. 
\label{cor:1}
\end{cor}
\begin{ex}
Consider a usual K\"ahler manifold as in Example \ref{ex:gkusual}. The sheaf $\Omega^\mathrm{ch}_M$ is isomorphic to the holomorphic chiral de Rham complex  constructed in \cite{malikov}. Indeed, 
in this case we have the that both hermitian complex structures agree $J=J^+=J^-$. We have natural isomorphisms then
\begin{equation}
\overline{L_1^+} \simeq T^{*0,1}, \qquad L_1^- \simeq T^{*1,0}
\label{eq:mamamamas}
\end{equation}
Recall from Proposition \ref{prop:diff1} that  $Q^+_0$ gets identified with with $\overline{\delta_+} = d_{L_1^+}$ which in this case is the usual operator $\overline{\partial}$.  If we exchange $\mathcal{J}_1$ and $\mathcal{J}_2$\footnote{This is to have the same signs as in the rest of this section} then we see that $\Omega^1 = T^*_\mathbb{C}$ is the holomorphic cotangent bundle, $\Theta^1=T_\mathbb{C}$ is the holomorphic Tangent bundle and $\mathscr{O}$ is the sheaf of holomorphic functions on $M$. Theorem \ref{thm:holomorphiccdr} gives an isomorphism of $\Omega^\mathrm{ch}_M$ with the holomorphic chiral de Rham complex constructed in \cite{malikov}. 
\end{ex}
\begin{rem}
In the usual K\"ahler case, we have by Theorem \ref{thm:holomorphiccdr} an embedding $\Omega^\mathrm{ch}_M \subset U^\mathrm{ch}(E)$. On the other hand, there is an obvious embedding of the holomorphic chiral de Rham complex of \cite{malikov} into the smooth chiral de Rham complex $U^\mathrm{ch}(E)$ given by the inclusions $T_\mathbb{C} \subset T\otimes \mathbb{C}$, $T^*_\mathbb{C} \subset T^* \otimes \mathbb{C}$ and $\mathscr{O} \subset C^\infty(M)$. These two inclusions are not the same. In fact they are an affine analog of the fact that the decomposition of cohomology given by \eqref{eq:8.3} is not the Dolbeaut decomposition but rather an orthogonal transformation of it \cite{gualtieri3}.
\label{rem:buenisimo}
\end{rem}
\begin{rem}
Since the zero mode $J^-_0$ is well defined on any generalized K\"ahler manifold, the sheaf $\Omega^\mathrm{ch}_M$ carries a grading given by fermionic charge. In the description of Theorem \ref{thm:holomorphiccdr} the sections of $\Omega^1$ have charge $+1$ while the sections of $\Theta^1$ have charge $-1$. 
\label{rem:cargas}
\end{rem}
\begin{rem}
One should view $(\mathcal{E}, \mathcal{D}, \langle \cdot, \cdot \rangle)$ as a Courant-Dorfman algebra over $\mathscr{O}$ \cite{roytenberg}. From this perspective, theorem \ref{thm:holomorphiccdr} is nothing more than the SUSY vertex algebra that one can trivially attach to any such algebra. 
\label{rem:royt}
\end{rem}
\begin{rem}
From Remark \ref{rem:cargas} and Corollary \ref{cor:1} it is natural to define a chiral analog of the Hodge decomposition as the sheaf cohomology 
\begin{equation}
H^{p,q}_\mathrm{ch}(M) = H^{q} \left( \Omega^{\mathrm{ch},p}_M \right),
\label{eq:9.9.1}
\end{equation}
where $\Omega^{\mathrm{ch},p}_M$ is the component of charge $p$ with respect to $J^-_0$. 

There are several subtleties in defining \eqref{eq:9.9.1}. First, we need to show that just as in the usual case, $\Omega^\mathrm{ch}_M$ carries a filtration such that the succesive quotients are (finitely presented) sheaves of $\mathscr{O}$-modules. This is done in parallel to the usual case. Having done so, one needs to develop the cohomology theory of sheaves of $\mathscr{O}$-modules on generalized K\"ahler manifolds. This is easily done in the bihermitian setup as explained below. We will not pursue this further in this article. We point however that Corollary \ref{cor:1} would imply a chiral version of the de Rham Theorem:
\begin{equation}
H^{p,q}_\mathrm{ch}(M) = \frac{\ker Q^+_{0}:U^\mathrm{ch}(E)^{p,q} \rightarrow U^{\mathrm{ch}}(E)^{p,q+1}}{\mathrm{im}\, Q^+_{0}: U^{\mathrm{ch}(E)^{p,q-1}} \rightarrow U^{\mathrm{ch}}(E)^{p,q}},
\label{eq:derham1}
\end{equation}
where $U^\mathrm{ch}(E)^{p,q}$ is consists of sections with charge $p$ (resp. $q$) with respect to $J^-_0$ (resp. $J^+_0$).
In the usual K\"ahler and non-supersymmetric case, one should compare this result to those of \cite{cheung}.
\end{rem}
The description of $\Omega^{\mathrm{ch}}_M$ in terms of ``holomorphic'' data is much easily done in the bihermitian setup. It will be useful to replace $\mathcal{J}_i$ with $-\mathcal{J}_i$ for simplicity (see the signs in Example \ref{ex:gkusual}).  We have isomorphisms \eqref{eq:b4.2}:
\begin{equation}
\overline{L_1} \simeq T^{*0,1}_+ \oplus T^{*0,1}_-,
\label{eq:10.1}
\end{equation}
Under this identification the differential $d_{L_1} = \overline{\partial_+} + \overline{\partial_-}$. It follows by definition then that we have natural isomorphisms
\begin{equation}
\mathscr{O} \simeq \mathscr{O}^+, \qquad \Omega^1 \simeq T^*_{\mathbb{C},+}, \qquad \Theta^1 \simeq T_{\mathbb{C},+}
\label{eq:10.2}
\end{equation}
where $\mathscr{O}^+$ (resp. $T^*_{\mathbb{C},+}$, $T_{\mathbb{C},+}$) is the sheaf of holomorphic functions (resp. holomorphic cotangent bundle, holomorphic tangent bundle) with respect to $J^+$. Recall (Prop. \ref{prop:4.2}) that $H$ is of type $(2,1) + (1,2)$ with respect to $J_+$. Let $H_+$ be its $(2,1)$ component. We see that
\begin{equation}
0 \rightarrow T^*_{\mathbb{C},+} \rightarrow \mathcal{E} \rightarrow T_{\mathbb{C},+} \rightarrow 0
\label{eq:10.3}
\end{equation}
is naturally a holomorphic exact Courant algebroid in the complex manifold $(M, J_+)$. It is the $H_+$-twisting of the standard Courant algebroid. We arrive at the following
\begin{prop}
Let $(M, \mathcal{J}_1, \mathcal{J}_2)$ be a generalized K\"ahler manifold, and let $(g,J_\pm)$ be the associated bihermitian data. The sheaf $\Omega^\mathrm{ch}_M$ is naturally isomorphic to the sheaf of SUSY vertex algebras $U^\mathrm{ch}(\mathcal{E})$ on the complex manifold\footnote{The sign change in $J_+$ is due to the sign changes in the previous paragraph.} $(M, -J_+)$ constructed as in Prop. \ref{prop:universal}, that is the holomorphic $H_+$-twisted \emph{chiral de Rham complex} of $(M, -J_+)$. 

Similarly, if we use $Q^-$ in place of $Q^+$ in \eqref{eq:8.8} we obtain a sheaf of SUSY vertex algebras on the complex manifold $(M, -J_-)$, its holomorphic twisted chiral de Rham complex. 
\label{prop:10.5}
\end{prop}
\begin{rem}
Note that we may in fact construct $4$-sheaves as half-twisted models if we take as BRST differentials $Q^\pm_0$ or $G^\pm_0$. The sheaves obtained using the latter differentials are the anti-holomorphic counterparts to the sheaves described in Prop \ref{prop:10.5}. In particular, in a usual generalized K\"ahler manifold as in Example \ref{ex:gkusual}, we have $J_+ = J_- = -J$ and therefore the sheaf computed with either $Q^+_0$ or $Q^-_0$ coincides with the usual holomorphic chiral de Rham complex of $M$ as in \cite{malikov} while the sheaf computed with either $G^+_0$ or $G^-_0$ is the anti-chiral de Rham complex of $M$. 
\label{rem:10.6}
\end{rem}
The following follows immediately from Theorem \ref{thm:1}:
\begin{prop}
Let $(M, \mathcal{J}_1, \mathcal{J}_2)$ be a generalized calabi-Yau metric manifold. Then the chiral de Rham complex $\Omega^\mathrm{ch}_M$ is a sheaf of topological vertex algebras, i.e. there exists an embedding of the $N=2$ superconformal vertex algebra into the space of global sections of $\Omega^\mathrm{ch}_M$. This algebra is generated by the cohomology classes of the fields $J^-$ and $H^-$ constructed in Theorem \ref{thm:1}.
\label{prop:8.3}
\end{prop}
\begin{rem}
Given this last proposition it is natural to define following \cite{borisov1} the \emph{2-variable Elliptic genus} of a compact generalized Calabi-Yau metric manifold $M$ as the supertrace
\begin{equation}
\mathcal{E}ll_M(y,q) := y^{- \frac{\dim M}{2}} \mathrm{st}_{H^* (\Omega(E))} y^{J^-_0}q^{L^-_0}.
\label{eq:8.9}
\end{equation}
\label{rem:8.4}
However, the no-go theorems (for a review see  \cite{witt1} and references therein)   state that if $M$ is compact then $H=0$, in which case the dilaton vanishes in the fields $J^\pm, H^\pm$ and we reduce to the usual Calabi-Yau case.
\end{rem}
We conclude this section by mentioning the different topological sectors of these sheaves \cite{heluani9}. Let $(M, \mathcal{J}_1, \mathcal{J}_2)$ be a generalized Calabi-Yau manifold. 
\begin{enumerate}
\item \textbf{B-twist} In this twist we consider the BRST charge $Q^B = Q^+ + Q^-$. The complex $(U^\mathrm{ch}(E), Q^B)$ is quasi-isomorphic to $(\wedge^\bullet \overline{L_1}, d_{L_1})$. If $M$ is a usual Calabi-Yau manifold with the generalized complex structures as in \ref{ex:gkusual}, the cohomology of this complex is given by the Hochschild cohomology of $M$: $\oplus H^{p,p} M$
\item \textbf{A-twist} Consider now the BRST differential  $Q^A = Q^+ + G^-$. In this case the complex $(U^\mathrm{ch}(E), Q^A)$ is quasi-isomorphic to $(\wedge^\bullet \overline{L_2}, d_{L_2})$ and it the usual Calabi-Yau case its cohomology is the de Rham cohomology of $M$.
\end{enumerate}

\section{Discussion} \label{sec:8}

In the present article we described explicitly an $N=2,2$ structure on the chiral-anti-chiral de Rham complex of a generalized Calabi-Yau metric manifold with or without $H$-flux. We also studied the half-twisted model, and showed that it can be described purely in terms of ``holomorphic'' data. In the course of doing so we clarified some results scattered in the literature on  generalized Calabi-Yau metric manifolds, as well as produced some new interesting ones. In particular, we showed in Section \ref{sec:5} how unimodularity of the Lie algebroids appearing in the generalized geometry side, corresponds to the properties of traces of connections on the bihermitian side.
  In the usual Calabi-Yau case, this corresponds to Ricci-flatness.

This article generates some questions which would be interesting to address in the future. 
\begin{itemize}
\item The analog of the Clifford or Dolbeaut decomposition of cohomology makes sense on any generalized complex manifold satisfying the $dd^c$ lemma \cite{cavalcanti}. Presumably one should be able to define the sheaf $\Omega^\mathrm{ch}_M$ in such cases. 
\item It would be interesting to match the discrepancy between the holomorphic CDR of \cite{malikov} and $\Omega^\mathrm{ch}_M$ as presented in Remark \ref{rem:buenisimo}, with the results of \cite{witten} and \cite{kapustin}. We suspect that this discrepancy corresponds to different choices of Hamiltonians. 
\item The notion of holomorphic Courant algebroid, or Courant-Dorfman algebra over $\mathscr{O}$ should be easy to define on any generalized K\"ahler manifold. To such objects one should be able to attach a vertex algebra in the same way as in the smooth case.  
\item Although the notion of Elliptic genus in two variables seems not to give anything new due to the no-go theorems of generalized complex geometry. Presumably one could develop a theory of cohomologies with compact support or similar, to make sense of \eqref{eq:8.9} in the non-compact case, where $H$-flux will play a crucial role. Note however that the results of \cite{cheung} should be easily generalized in this setting to the generalized K\"ahler case. 
\end{itemize}

\appendix
\section{SUSY conformal algebras} \label{sec:appendixc}
Let $\mathfrak{h}$ be super Lie algebra spanned by an odd element $\chi$ and an even element $\lambda$ such that $[\chi,\chi] = -2\lambda$ and $\lambda$ is central. Let $\cH$ be its universal enveloping algebra. We will consider another set of generators $S, -T$ for the same algebra. 
\begin{defn}
     \emph{An $N_K=1$ SUSY Lie
    conformal  algebra} is a $\cH$-module $\cR$ with an operation
    $[\,_{\Lambda}\,]: \cR \otimes \cR \rightarrow \cH
    \otimes \cR$ of degree
    $1$ satisfying:
    \begin{enumerate}
        \item Sesquilinearity
            \begin{equation*}
                [S a_\Lambda b] =  \chi [a_\Lambda b]
                \qquad [a_\Lambda S b] = -(-1)^{a} \left(S
                + \chi
                \right) [a_\Lambda b]
            \end{equation*}
        \item Skew-Symmetry:
            \begin{equation*}
                [b_\Lambda a] =  (-1)^{a b} [b_{-\Lambda -
                \nabla} a]
            \end{equation*}
            Here the bracket on the right hand side is computed as
            follows: first compute $[b_{\Gamma}a]$, where $\Gamma =
            (\gamma, \eta)$ are generators of $\cH$ super commuting
            with $\Lambda$, then replace $\Gamma$ by $(-\lambda - T,
            -\chi - S)$.
        \item Jacobi identity:
            \begin{equation*}
                [a_\Lambda [b_\Gamma c]] = -(-1)^{a} \left[
                [ a_\Lambda b]_{\Gamma + \Lambda} c \right] +
                (-1)^{(a+1)(b+1)} [b_\Gamma [a_\Lambda c]]
            \end{equation*}
            where the first bracket on the right hand side is computed as in Skew-Symmetry
            and the identity is an identity in $\cH^{\otimes 2} \otimes\cR$.
    \end{enumerate}
    \label{defn:k.conformal.1}
    \end{defn}
    Given an $N_K=1$ SUSY VA, it is canonically an $N_K=1$ SUSY Lie conformal algebra with
    the bracket defined in (\ref{eq:2.4.2}). Moreover, given an $N_K=1$ Lie
    conformal algebra $\cR$, there exists a unique $N_K=1$ SUSY VA called the
    \emph{universal enveloping SUSY vertex algebra of $\cR$} with the property
    that if $W$ is another $N_K=1$ SUSY VA and $\varphi : \cR \rightarrow W$ is a
    morphism of Lie conformal algebras, then $\varphi$ extends uniquely to a
    morphism $\varphi: V \rightarrow W$ of SUSY VAs.
    The operations (\ref{eq:2.4.2}) satisfy:
    \begin{itemize}
        \item Quasi-Commutativity:
            \begin{equation*}
                ab - (-1)^{ab} ba = \int_{-\nabla}^0 [a_\Lambda
                b] d\Lambda
            \end{equation*}
        \item Quasi-Associativity
            \begin{equation*}
                (ab)c - a(bc) = \sum_{j \geq 0}
                a_{(-j-2|1)}b_{(j|1)}c + (-1)^{ab} \sum_{j \geq 0}
                b_{(-j-2|1)} a_{(j|1)}c
            \end{equation*}
        \item Quasi-Leibniz (non-commutative Wick formula)
            \begin{equation*}
                [a_\Lambda bc ] = [a_\Lambda b] c + (-1)^{(a+1)b}b
                [a_\Lambda c] + \int_0^\Lambda [ [a_\Lambda
                b]_\Gamma c] d \Gamma
            \end{equation*}
    \end{itemize}
    where the integral $\int d\Lambda$ is $\partial_\chi \int d\lambda$. In
    addition, the vacuum vector is a unit for the normally ordered product
    and the endomorphisms $S, T$ are odd and even derivations respectively of
    both operations.

\end{document}